\documentclass[a4paper, 11pt, twoside, leqno]{article}

\usepackage[T1]{fontenc}
\usepackage[utf8x]{inputenc}
\usepackage[english]{babel}
\usepackage{lmodern}               
\usepackage{amsmath,amsthm,amssymb}
\usepackage{mathrsfs,esint,bbm,stmaryrd}
\usepackage[shortlabels]{enumitem}
\usepackage{graphicx,xcolor,subcaption}
\usepackage[normalem]{ulem}
\usepackage[textwidth=16.1cm,centering,headheight=24pt]{geometry}
\usepackage{authblk}
\usepackage{tikz}
\usetikzlibrary{calc,intersections,through}
\usetikzlibrary{backgrounds}

\definecolor{lightblue}{RGB}{150,180,255}
\colorlet{textblue}{lightblue!35!blue!60!black}

\newtheorem{theorem}{Theorem}           

\newtheorem{lemma}[theorem]{Lemma}
\newtheorem{prop}[theorem]{Proposition}

\theoremstyle{definition}              

\theoremstyle{remark}                  
\newtheorem{step}{Step}

\newtheorem{remark}{Remark}

\DeclareMathOperator{\tr}{tr} 
\DeclareMathOperator{\dist}{dist} 
\DeclareMathOperator{\sign}{sign}

\let\div\relax
\DeclareMathOperator{\div}{div} 
\DeclareMathOperator{\BV}{BV}
\DeclareMathOperator{\SBV}{SBV}

\DeclareMathOperator{\loc}{loc}

\newcommand{\abs}[1]{\left| #1 \right|}
\newcommand{\norm}[1]{\left\| #1 \right\|}

\newcommand{\csubset}{\subset\!\subset}

\DeclareMathAlphabet{\mathpzc}{OT1}{pzc}{m}{it}

\newcommand{\J}{\mathrm{J}}
\renewcommand{\d}{\mathrm{d}}
\renewcommand{\o}{\mathrm{o}}
\newcommand{\R}{\mathbb{R}}
\newcommand{\Z}{\mathbb{Z}}
\newcommand{\C}{\mathbb{C}}
\renewcommand{\S}{\mathbb{S}}

\newcommand{\Q}{\mathbf{Q}}
\newcommand{\M}{\mathbf{M}}
\newcommand{\I}{\mathbf{I}}
\newcommand{\n}{\mathbf{n}}
\newcommand{\m}{\mathbf{m}}
\newcommand{\q}{\mathbf{q}}
\renewcommand{\u}{\mathbf{u}}
\newcommand{\NN}{\mathscr{N}}

\newcommand{\EE}{\mathscr{E}}
\renewcommand{\H}{\mathscr{H}}
\renewcommand{\L}{\mathscr{L}}

\newcommand{\F}{\mathscr{F}}

\newcommand{\eps}{\varepsilon}
\newcommand{\ess}{{\mathrm{ess}}}

\SetSymbolFont{stmry}{bold}{U}{stmry}{m}{n}  
\newcommand{\nnu}{{\boldsymbol{\nu}}}

\newcommand{\Sz}{\mathscr{S}_0^{2\times 2}}
\newcommand{\bd}{{\mathrm{bd}}}
\newcommand{\Qb}{\Q_{\bd}}

\definecolor{lightblue}{rgb}{0.22,0.45,0.70}   
\definecolor{darkgray}{gray}{0.4}    
\definecolor{lightgray}{gray}{0.8}

\title{Liftings of Sobolev maps into closed Riemannian manifolds via double coverings and minimal connections relative to planar sets, with an application to ferronematics}

\date{}
\author{Giacomo~Canevari, Federico~Luigi~Dipasquale, Bianca~Stroffolini}

\newcommand{\Addresses}{{
  \bigskip
  \footnotesize

  Giacomo~Canevari \\
  \textsc{Universit\`{a} di Verona} \\
  Strada Le Grazie 15, 37134 Verona, Italy \\
  \textit{E-mail address}: \texttt{giacomo.canevari@univr.it}

  \medskip

  Federico~Luigi~Dipasquale \\
  \textsc{Scuola Superiore Meridionale}\\
  Via Mezzocannone 4, 80138 Napoli, Italy\\
  \textit{E-mail address}: \texttt{f.dipasquale@ssmeridionale.it}

  \medskip

  Bianca~Stroffolini \\
  \textsc{Dipartimento di Matematica e Applicazioni ``Renato Caccioppoli'',}\\
  \textsc{Universit\`{a} degli studi di Napoli ``Federico II''}\\
  Via Cintia, Monte S. Angelo, 80126 Napoli, Italy\\ 
  \textit{E-mail address}: \texttt{bstroffo@unina.it}
}}

\begin{document}

\maketitle

\begin{abstract}
We consider Sobolev maps from a planar domain into a closed Riemannian manifold and their BV liftings via a double covering of the target. We establish a sharp lower bound on the jump length of the lifting, expressed in terms of a geometric quantity: the minimal connection, relative to the domain, of the non-orientable singularities. As an application, we analyse minimisers of a two-dimensional model of ferronematics under ``mixed'' boundary conditions --- that is, Dirichlet conditions for the liquid crystal order parameter and Neumann conditions for the magnetisation vector.


\end{abstract}

Let~$\NN$, $\EE$ be closed Riemannian manifolds and~$\Pi\colon\EE\to\NN$ a covering map. The lifting problem asks whether a given map $u\colon\Omega \to\NN$, defined in a domain~$\Omega\subseteq\R^n$, can be factorised as~$u = \Pi\circ v$ for some lifting~$v\colon\Omega\to\EE$ with comparable regularity. While the problem is classical for continuous maps on simply connected domains, its analysis becomes substantially more delicate in weaker regularity settings, such as Sobolev or BV spaces.
The most studied case is the covering map $\Pi\colon\R\to\S^1$ given by~$\Pi(\theta) := e^{i\theta}$, which arises naturally, e.g., in the Ginzburg--Landau theory of superconductors. The analysis of this case was initiated in \cite{BethuelZheng, BethuelDemengel} and later completed by Bourgain, Brezis, and Mironescu~\cite{BourgainBrezisMironescu2005}, who completely characterised the fractional Sobolev spaces~$W^{s,p}(\Omega, \, \S^1)$ in which the lifting problem  always has a positive answer (see also~\cite{BrezisMironescu} for a comprehensive discussion).

Subsequent contributions extended these results to Besov spaces~\cite{MironescuRussSire} and to more general target manifolds~\cite{BethuelChiron, BallZarnescu, Mucci-DCDS, MironescuVanSchaftingen}. Even in the familiar setting of $W^{1,p}$ spaces, not all $W^{1,p}$-maps admit a $W^{1,p}$-lifting — at least when~$1 < p < 2$; however, all maps of bounded variation admit a lifting of bounded variation.
For liftings through the exponential map~$\Pi\colon\R\to\S^1$, this fact was first proved by Giaquinta, Modica, and Sou\v{c}ek~\cite{GiaquintaModicaSoucek} and later revisited by Davila and Ignat~\cite{DavilaIgnat} and Merlet~\cite{Merlet}, using different proofs. For liftings through the quotient map~$\Pi\colon\S^d\to\R\mathrm{P}^d$, which defines real projective spaces, this was proved by Ignat and Lamy~\cite{IgnatLamy}, with earlier (partial) results by Bedford~\cite{Bedford}.
Generalisations to a wider class of covering maps~$\Pi\colon\EE\to\NN$ were given in~\cite{CO-lifting} and~\cite{ContiCrismaleGarroniMalusa}; see also~\cite{BellettiniMarzianiScala} for the existence of liftings in GSBV
of circle-valued maps.

In this paper, we restrict our attention to maps defined on planar domains, $\Omega\subseteq\R^2$, and focus on double coverings~$\Pi\colon\EE\to\NN$ (see Section~\ref{sect:lifting-prop} below for a definition). The prototypical example we have in mind is the double covering~$\Pi\colon\S^1\to\S^1$, given in complex notation by~$\Pi(z) := z^2$. This map arises naturally in the modelling of two-dimensional (nematic) liquid crystals, which are anisotropic fluids whose constituent molecules exhibit long-range orientational order while remaining able to flow.
Under suitable conditions ensuring in-plane molecular alignment, the state of a thin nematic layer can be represented by a map from a physical domain~$\Omega\subseteq\R^2$ to the set~$\NN$ of $2\times 2$ real symmetric matrices~$\Q$ satisfying~$\tr\Q := Q_{ii} = 0$ and~$\abs{\Q}^2 := Q_{ij} Q_{ij} = 1$. Physically, the eigenspace of~$\Q$ associated with its unique positive eigenvalue encodes the preferred (non-oriented) molecular direction at each point.
In this example, the set~$\NN$ is a smooth, compact manifold, diffeomorphic to the unit circle~$\S^1\subseteq\C$. A lifting, or an orientation, for~$\Q\colon\Omega\to\NN$
is a vector field~$\n\colon\Omega\to\S^1$ that satisfies
\begin{equation} \label{lifting}
 \Q(x) =\sqrt{2}\left(\n(x)\otimes\n(x) - \frac{\I}{2}\right)
\end{equation}
for a.e.~$x\in\Omega$. Such a vector field generates the eigenspace of~$\Q$ corresponding to the positive eigenvalue at almost every point.
From a mathematical viewpoint, and up to composition with a diffeomorphism~$\S^1\to\NN$, a lifting~$\M$ of~$\Q$ is a selection of a square root of~$\Q$.

More generally, given a double covering~$\Pi\colon\EE\to\NN$ between closed Riemannian manifolds and maps~$u\colon\Omega\to\NN$, $v\colon \Omega\to\EE$, we will say that~$v$
is a lifting of~$u$ via~$\Pi$ if and only if~$u(x) = \Pi(v(x))$ for a.e.~$x\in\Omega$.
The results in~\cite{GiaquintaModicaSoucek, DavilaIgnat, IgnatLamy, CO-lifting, ContiCrismaleGarroniMalusa} imply that each map~$u\in W^{1,1}(\Omega, \, \NN)$ admits a lifting~$v\in\SBV(\Omega, \, \EE)$, although not necessarily in~$W^{1,1}(\Omega, \, \EE)$. Our first result provides a lower bound for the length of the jump set of~$v$ in terms of the non-orientable singularities of~$u$ and the minimal connection between them.
Let~$a_1$, \ldots $a_p$ be distinct points in~$\Omega$.
We define a \emph{connection for
$\{a_1, \, \ldots, \, a_{p}\}$ relative to~$\Omega$}
as a finite collection of closed, non-degenerate
straight line segments $\{L_1, \, \ldots, \, L_q\}$
with the following properties:
\begin{enumerate}[label=(\roman*)]
 \item each~$L_j$ is contained in~$\overline{\Omega}$;
 \item for each~$j$, either~$L_j$ connects two of the
 points~$a_1, \, \ldots, \, a_p$ or~$L_j$ connects
 one of the points~$a_i$ with a point of~$\partial\Omega$;
 \item for each~$i$, there is an
 \emph{odd} number of indices~$j$ such
 that~$a_i$ is an endpoint of~$L_j$.
\end{enumerate}
An example is given in Figure~\ref{fig:connection}.
We define
\begin{equation} \label{minconnrel}
 \begin{split}
  \mathbb{L}^\Omega(a_1, \, \ldots, \, a_p)
  := \min\bigg\{
   \sum_{j = 1}^q \H^1(L_j)\colon
   &\{L_1, \, \ldots, \, L_q\} \textrm{ is a connection}\\[-3mm]
   &\textrm{for }
   \{a_1, \, \ldots, \, a_p\}
   \textit{ relative to } \Omega \bigg\} .
 \end{split}
\end{equation}
If~$\{L_1, \, \ldots, \, L_q\}$ is a minimiser for
the right-hand side of~\eqref{minconnrel},
we will say that~$\{L_1, \, \ldots, \, L_q\}$
is a minimal connection for~$\{a_1, \, \ldots, \, a_p\}$
relative to~$\Omega$.
This notion of minimal connection is reminiscent of the one introduced by Brezis, Coron, and Lieb~\cite{BrezisCoronLieb}, with the key difference that the minimal connection in~\cite{BrezisCoronLieb} is \emph{oriented}, whereas ours is not. In that setting, the points~$a_i$ are assigned the labels~$1$ or~$-1$, and line segments must connect a positive point with a negative one; in contrast, in our framework, all line segments joining points in~$a_i$ are admissible. A similar concept of non-oriented minimal connection (or minimal connection modulo two) was considered in~\cite{CanevariMajumdarStroffoliniWang}, although there, line segments joining a point in~$a_i$ with~$\partial\Omega$ were not allowed, a fact that in \cite{CanevariMajumdarStroffoliniWang} came as a consequence of the specific Dirichlet boundary conditions considered there, ensuring in particular that $p$ was even, coupled with a convexity assumption on $\Omega$. 

\begin{figure}[t]
 \centering
 \resizebox{.45\textheight}{!}{%
 \begin{tikzpicture}

\tikzset{
    point/.style={circle, fill=textblue, inner sep=0pt, minimum size=3pt},
    connection/.style={textblue, thick},
    potato/.style={fill=lightblue, thick}
}

\begin{scope}[on background layer]
    \draw[potato] 
        (0,0) node[above] {$\Omega$}
          .. controls (1,1) and (2,0.5) .. 
        (3,0.8) .. controls (4,1) and (4.5,0) .. 
        (4,-1) .. controls (3.5,-2) and (2.5,-1.8) .. 
        (1.5,-1.5) .. controls (0.5,-1.2) and (-0.5,-0.8) .. 
        cycle;
\end{scope}

\node[point, label=left: $a_1$] (p1) at (1,0) {};
\node[point, label=right:$a_5$] (p2) at (3,0) {};
\node[point, label=right:$a_3$] (p3) at (2,-1) {};
\node[point, label=left: $a_2$] (p4) at (0.8,-0.7) {};
\node[point, label=right:$a_6$] (p5) at (2,0.5) {};
\node[point, label=right:$a_4$] (p6) at (3.5,-0.5) {};

\draw[connection] (p1) -- (p5);
\draw[connection] (p1) -- (p4);
\draw[connection] (p1) -- (p6);
\draw[connection] (p2) -- (3, 0.8);
\draw[connection] (p3) -- (1.5,-1.5); 
\end{tikzpicture}
}
 \caption{Example of a connection for~$\{a_1 \, \ldots, a_6\}$
 relative to~$\Omega$.}
 \label{fig:connection}
\end{figure}
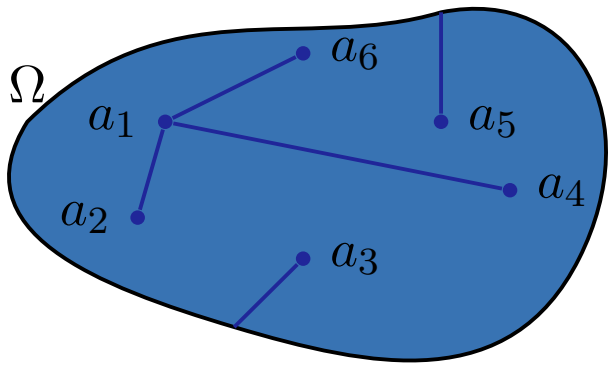

In the next result, we denote by~$\J_{v}$ the jump set of a map~$v\in\SBV(\Omega, \, \R^2)$, and we write $\H^1$ for the $1$-dimensional Hausdorff measure (i.e., length). Given two sets~$A, B \subseteq \R^2$, we write $A = B \mod \H^1$ if and only if $\H^1(A\setminus B) + \H^1(B\setminus A) = 0$. Our statement applies to maps $u\colon\Omega\to\NN$ which are locally~$W^{1,2}$ away from a finite set of points singularities, which may be orientable --- meaning, roughly speaking, that $u$ admits a Sobolev lifting in a neighbourhood of that singularity --- or not. For a more precise definition of orientable and non-orientable singularities, see Section~\ref{sect:lifting-prop}. However, only the non-orientable singularities play a r\^ole in the main estimate~\eqref{lowerboundSM}.

\begin{theorem} \label{th:minconnrel}
 Let~$\NN$, $\EE$ be closed, smooth Riemannian manifolds and~$\Pi\colon\EE\to\NN$
 a smooth double covering map.
 Let~$\Omega\subseteq\R^2$ be a bounded, simply connected domain
 of class~$C^2$ and let~$a_1$, \ldots, $a_p$,
 $b_1$, \ldots, $b_r$ be distinct points in~$\Omega$. Let
 \[
  u\in W^{1,1}(\Omega, \, \NN)\cap
  W^{1,2}_{\loc}(\Omega\setminus
  \{a_1, \, \ldots, a_p, \, b_1, \, \ldots, \, b_r\}, \, \NN)
 \]
 be a map with a non-orientable singularity at each~$a_j$
 and an orientable singularity at each~$b_h$.
 If~$v\in\SBV(\Omega, \, \EE)$ is a lifting
 for~$\EE$ via~$\Pi$, then
 \begin{equation} \label{lowerboundSM}
  \H^1(\J_{v}) \geq \mathbb{L}^\Omega(a_1, \, \ldots, \, a_p).
 \end{equation}
 The equality holds if and only if there exists a minimal
 connection~$\{L_1, \, \ldots, \, L_q\}$
 for~$\{a_1, \, \ldots, a_p\}$ relative to~$\Omega$ such that
 $\J_v = \bigcup_{j=1}^q L_j \mod\H^1$.
\end{theorem}

Theorem~\ref{th:minconnrel} generalises Proposition~A.1 in~\cite{CanevariMajumdarStroffoliniWang}, which holds under the additional assumptions that~$\Omega$ is convex and~$v$ is continuous in a neighbourhood of~$\partial\Omega$. Heuristically, Theorem~\ref{th:minconnrel} follows from the topological observation that boundary points of the jump set~$\J_{v}$ of a lifting can only be non-orientable singularities of~$u$ or points in~$\partial\Omega$. However, defining precisely the ``boundary of the jump set~$\J_{v}$'' requires care, as the standard notion of topological boundary is inadequate and rectifiable sets, without additional structure, do not come with a natural boundary operator. A more suitable notion in this context is the boundary modulo two in the sense of currents, as used in~\cite{CanevariMajumdarStroffoliniWang}. Here, we avoid mentioning explicitly the theory of currents and instead rely on a result by Ambrosio, Caselles, Masnou, and Morel~\cite{ACMM} on the structure of planar sets of finite perimeter.
However, the arguments still rely implicitly on the homological structure of currents modulo two, as becomes most apparent in Lemma~\ref{lemma:boundary}.

In case~$\Pi$ is the double covering~$\S^1\to\S^1$, Theorem~\ref{th:minconnrel} may be of potential interest in some contexts related to two-dimensional modelling of  liquid crystals. Among the possible applications, we mention the discrete model for string defects in nematics studied in~\cite{Badal_et_al} and the Ginzburg--Landau model with topological free discontinuities in~\cite{GoldmanMerletMillot, BadalCicalese} (although the latter also applies to covering maps of arbitrary finite order).
Here, however, we focus on an application to \emph{ferronematics}.
Ferronematics are composite materials obtained by
suspending magnetic nanoparticles in a nematic liquid crystal host~\cite{Brochard,MLDC}.
We adopt the modelling approach of~\cite{bisht2019},
based on two order parameters.
The orientation of the liquid crystal molecules
is described by the Landau-de Gennes $\Q$-tensor,
which is a map from the physical domain~$\Omega\subseteq\R^2$
to the space~$\Sz$ of~$2\times 2$, symmetric, real matrices
with trace equal to zero.
Nonzero values of~$\Q$ correspond to configurations
with a well-defined molecular alignment,
while~$\Q = 0$ indicates an isotropic state,
where all the directions of molecular alignment are equally likely.
The distribution of magnetic nanoparticles is
described by the average magnetisation vector,
$\M\colon\Omega\to\R^2$.
The system is described by the free-energy functional
\begin{equation}\label{eq:Feps}
 \F_\eps(\Q, \, \M) := \int_G \left(\frac{1}{2} \abs{\nabla \Q}^2 + \frac{\eps}{2} \abs{\nabla \M}^2 + \frac{1}{\eps^2} f_\eps(\Q,\M)\right) \d x,
\end{equation}
where~$\eps$ is a non-dimensional parameter.
The interaction between the liquid crystal
host and the magnetic inclusions is encoded in the potential~$f_\eps$,
which takes the form
\begin{equation}\label{eq:potential}
 f_\eps(\Q, \, \M) := \frac{1}{4}\left(1-\abs{\Q}^2\right)^2
+ \frac{\eps}{4}\left(1-\abs{\M}^2\right)^2 - \eps \beta \, \Q \M \cdot \M + \kappa_\eps.
\end{equation}
Here~$\beta> 0$ is given and~$\kappa_\eps$ is an additive
constant, depending only on~$\eps$ and~$\beta$,
chosen so that~$\inf f_\eps = 0$.
Potential energies similar to~\eqref{eq:potential}
arise from suitable homogenisation limits~\cite{Caldereretal}.
For positive values of~$\beta$, the potential~$f_\eps$
favours alignment between the liquid crystal
molecules and the magnetisation vector.
Indeed, the potential~$f_\eps(\Q, , \M)$ is minimised
by pairs~$(\Q_{\eps}^{\mathrm{pot}}, \, \M_{\eps}^{\mathrm{pot}})$
that satisfy the conditions
\begin{equation} \label{minpot}
	|\M_{\eps}^{\mathrm{pot}}|
	= \lambda_{\eps,  \beta}, \qquad
	\Q_{\eps}^{\mathrm{pot}}
	= \sqrt{2} s_{\eps, \beta}\left( \frac{\M_{\eps}^{\mathrm{pot}} \otimes \M_{\eps}^{\mathrm{pot}}}{\lambda_{\eps,\beta}^2} - \frac{\I}{2}  \right)
\end{equation}
where~$\mathbf{I}$ is the~$2\times 2$ identity matrix
and~$\lambda_{\eps, \beta}$, $s_{\eps,\beta}$ are
positive constants, uniquely determined by~$\eps$ and~$\beta$,
such that
\[
 \lambda_{\eps,\beta} \to\left(\sqrt{2}\beta + 1\right)^{1/2},
 \qquad s_{\eps,\beta} \to 1
\]
as~$\eps\to 0$ (see~\cite[Lemma~B.2]{CanevariMajumdarStroffoliniWang}).
We study the limit as~$\eps\to 0$, which is physically
relevant in the ``large domain'' regime
(i.e., the size of the domain is much larger than
the typical correlation length for the liquid crystal molecules).
The asymptotic analysis of Ginzburg-Landau-type functionals shows that the energy of the~$\Q$-component concentrates at finitely many points,
which correspond to \emph{non-orientable} singularities
of the limiting liquid crystal configuration.
In particular, the $\Q$-component of minimisers
converges to a matrix-valued map that does not
admit a continuous orthonormal eigenframe.
However, as the coupling promotes alignment between~$\M$
and the eigenvectors of~$\Q$, the energy for the~$\M$-component
is expected to concentrate along singular lines, corresponding to jumps in the eigenvector frame.


The asymptotic analysis of minimisers for~$\F_\eps$,
subject to Dirichlet boundary conditions for
both~$\Q_\eps$ and~$\M_\eps$,
has been carried out in the paper~\cite{CanevariMajumdarStroffoliniWang}.
We will consider here mixed boundary conditions,
i.e. Dirichlet boundary conditions for~$\Q_\eps$
and homogeneous Neumann boundary
conditions for~$\M_\eps$:
\begin{equation} \label{bcbis}
 \Q_\eps = \Qb, \quad \partial_\nnu\M_\eps = 0
 \qquad \textrm{on } \partial\Omega,
\end{equation}
where~$\nnu$ is the exterior unit normal to~$\partial\Omega$.
We then assume the boundary datum~$\Qb$
(does not depend on~$\eps$ and) takes the form
\begin{equation} \label{hp:bcbis}
 \Qb = \sqrt{2}\left(\n_{\bd}\otimes\n_{\bd}
  - \frac{\mathbf{I}}{2}\right)
 \qquad \textrm{on } \partial\Omega,
\end{equation}
for some map~$\n_{\bd}\in C^2(\partial\Omega, \, \R^2)$
which, a priori, is completely independent of the
values of~$\M_\eps$ on the boundary.
We assume that the domain~$\Omega\subseteq\R^2$ is bounded,
simply connected and of class~$C^2$, so that
$\n_{\bd}\colon\partial\Omega\to\S^1$
has a well-defined topological degree, which we call~$d\in\mathbb{Z}$.

In the following, we denote by $(\Q^\star_\eps, \, \M^\star_\eps)$ a minimiser of the functional~\eqref{eq:Feps} subject to the boundary conditions~\eqref{bcbis}. Given distinct points~$a_1, \ldots, a_{2\abs{d}}$ in~$\Omega$, we write~$\mathbb{W}(a_1, \ldots, a_{2\abs{d}})$ for their Ginzburg--Landau renormalised energy, as defined by Bethuel, Brezis, and H\'elein~\cite{BBH}
(see Equation~\eqref{W} in Section~\ref{sect:minimisers} for the definition).

\begin{theorem} \label{th:ferronematics}
 Let~$\Omega\subseteq\R^2$ be a bounded, simply connected
 domain of class~$C^2$. Then, there exists a (non-relabelled) subsequence,
 maps~$\Q^\star\colon\Omega\to\NN$, $\M^\star\colon\Omega\to\R^2$
 and distinct points~$a^\star_1, \, \ldots, \, a^\star_{2\abs{d}}$
 in~$\Omega$ such that the following holds:
 \begin{enumerate}[label=(\roman*)]
  \item $\Q^\star_\eps\to\Q^\star$ strongly in~$W^{1,p}(\Omega)$
  for any~$p < 2$;
  \item $\M^\star_\eps\to\M^\star$ strongly in~$L^p(\Omega)$
  for any~$p<+\infty$;
  \item $\Q^\star$ has a non orientable singularity
  at each point~$a^\star_j$ and satisfies
  \[
   \partial_j \left(Q^\star_{11} \, \partial_j Q^\star_{12}
    - Q^\star_{12} \, \partial_j Q^\star_{11}\right) = 0
  \]
  in the sense of distributions in~$\Omega$;
  \item $\M^\star\in\SBV(\Omega, \, \R^2)$, it satisfies
  $\abs{\M^\star} = (\sqrt{2}\beta + 1)^{1/2}$ a.e.~in~$\Omega$,
  and~$(\sqrt{2}\beta + 1)^{-1/2} \M^\star$ is a lifting of~$\Q^\star$;
  \item there exists a minimal connection~$(L_1, \, \ldots, \, L_{\abs{d}})$
  for~$(a^\star_1, \, \ldots, \, a^\star_{2\abs{d}})$
  relative to~$\Omega$ such that
  $\J_{\M^\star} = \bigcup_{j=1}^q L_j \mod\H^1$;
  \item $(a^\star_1, \, \ldots, \, a^\star_{2\abs{d}})$
  minimises the function
  \[
   \mathbb{W}_\beta^{\Omega}(a_1, \, \ldots, \, a_{2\abs{d}}) :=
    \mathbb{W}(a_1, \, \ldots, \, a_{2\abs{d}})
    + \frac{2\sqrt{2}}{3} \left(\sqrt{2}\beta + 1\right)^{3/2}
     \mathbb{L}^{\Omega}(a_1, \, \ldots, \, a_{2\abs{d}})
  \]
  among all the $(2\abs{d})$-uples~$(a_1, \, \ldots, \, a_{2\abs{d}})$
  of distinct points in~$\Omega$.
 \end{enumerate}
\end{theorem}

Theorem~\ref{th:ferronematics} is a variant of Theorem~1 in \cite{CanevariMajumdarStroffoliniWang}.  
More generally, items~{(i)}, {(ii)}, {(iv)}, and a slightly weaker variant of item~{(iii)} 
of Theorem~\ref{th:ferronematics} are true for any sequence of critical points with potential energy 
equibounded with respect to $\eps$, under both pure Dirichlet boundary conditions as in \cite{CanevariMajumdarStroffoliniWang} and `mixed' boundary conditions as in~\eqref{bcbis}, \eqref{hp:bcbis}. 
This comes as an outcome of the PDE analysis carried out in the companion paper~\cite{CDS}, 
which leads also to improved convergence results for both $\Q$ and $\M$ and to the characterisation 
of the singular set of the $\M$-part in terms of $\H^1$-rectifiable varifolds with generalised curvature 
supported on the singular set of the $\Q$-part, the latter being a finite set of points even 
in this more general context.
We further note that point defects and the line defects connecting them also appear in other variational models, such as~\cite{Badal_et_al, GoldmanMerletMillot, BadalCicalese}. However, the mathematical nature of these problems differs: in our case, point and line defects emerge from a nontrivial coupling of two order parameters, whereas in~\cite{GoldmanMerletMillot, BadalCicalese} they arise from a free discontinuity problem with a topological constraint on the jumps, and in~\cite{Badal_et_al} from a discrete variational problem with a specific structure.

The rest of the paper is organised as follows. Section~\ref{sect:lifting} is devoted to the proof of Theorem~\ref{th:minconnrel}.
More precisely, after recalling some preliminary material on covering maps and Sobolev liftings (Subsection~\ref{sect:lifting-prop}) and proving a few key properties of minimal connections (Subsection~\ref{sect:minrelconn}), we reformulate Theorem~\ref{th:minconnrel} as a statement on sets of finite perimeter, namely Proposition~\ref{prop:minconn-finiteper} in Subsection~\ref{sect:lifting-reformulate}. Along the way, we recall definitions and properties of sets of finite perimeter, as needed. In Subsection~\ref{sect:minconn-finiteper}, we prove Proposition~\ref{prop:minconn-finiteper} using the results of~\cite{ACMM}. Section~\ref{sect:minimisers} is devoted to the proof of Theorem~\ref{th:ferronematics}.

\setcounter{equation}{0}
\numberwithin{equation}{section}
\numberwithin{definition}{section}
\numberwithin{theorem}{section}
\numberwithin{remark}{section}
\numberwithin{example}{section}

\section{Liftings and relative minimal connections}
\label{sect:lifting}


\subsection{Double coverings of Riemannian manifolds}
\label{sect:lifting-prop}

{First, let us recall the definition and basic properties
of (double) covering maps. We refer the reader to,
e.g., \cite{Hatcher} for more details.
Let~$\NN$, $\EE$ be two closed (i.e., compact and without boundary)
smooth Riemannian manifolds. A smooth map~$\Pi\colon\EE\to\NN$
is called a double covering of~$\NN$
if and only if for any~$w\in\NN$ there exists
an open neighbourhood~$V$ of~$w$ and two disjoint
open sets~$U_1$, $U_2\subseteq\EE$
such that~$\Pi^{-1}(V) = U_1 \cup U_2$ and the restrictions
$\Pi_{|U_1}\colon U_1\to V$, $\Pi_{|U_2}\colon U_2\to V$ are
isometric diffeomorphisms. It follows from this definition
that the inverse image~$\Pi^{-1}(w)$ of any point~$w\in\NN$
contains exactly two point of~$\EE$.

\begin{remark} \label{rk:uniform}
 Since~$\NN$ is compact, it is possible
 to choose~$V$, $U_1$, $U_2$ in the definition of covering
 to be geodesic balls with uniform radius. More precisely,
 let~$B_0 := B_{\delta_0}^{\NN}(w)$ be the geodesic open ball in~$\NN$
 of centre~$w$ and radius~$\delta_0$.
 There exists~$\delta_0 > 0$ small enough that,
 for any~$w\in\NN$, the inverse image~$\Pi^{-1}(B_0)$
 consists exactly of two disjoint geodesic balls in~$\EE$,
 say~$B_1 = B_{\delta_0}^{\EE}(z_1)$ and~$B_2 := B_{\delta_0}^{\EE}(z_2)$,
 and the restrictions~$\Pi_{|B_1}\colon B_1\to B_0$,
 $\Pi_{|B_2}\colon B_2\to B_0$ are isometric diffeomorphisms.
 This follows from the definition of double covering,
 reasoning e.g.~by contradiction and using the compactness of~$\NN$.
\end{remark}

The prototypical example of a double covering is the
natural quotient map $\Pi\colon\S^{n}\to\R\mathrm{P}^n$,
whose restriction to any open set~$U\subseteq\S^n$ strictly
contained in a half-sphere is a one-to-one isometry.
When~$n = 1$, the real projective line~$\R\mathrm{P}^1$
is (isometrically) diffeomorphic to the unit circle and,
up to composition with this diffeomorphism,
$\Pi$ can be identified with the map
defined by~$\Pi(z) := z^2$ for~$z\in\S^1\subseteq\C$.

For any double covering~$\Pi\colon\EE\to\NN$, there exists a
unique isometric diffeomorphism~$\Phi\colon\EE\to\EE$ such that
\begin{equation} \label{Phi}
 \Pi(\Phi(z)) = z, \qquad \Phi(z)\neq z \qquad \textrm{for all } z\in\EE.
\end{equation}
(For instance, in case~$\Pi\colon\S^n\to\R\mathrm{P}^n$,
$\Phi\colon\S^n\to\S^n$ is the antipodal map, $\Phi(z) := -z$.)
For each~$z\in\EE$, $\Phi(z)$ is defined as the unique
point~$y\in\EE$ such that~$y\neq z$ and~$\Pi(y) = \Pi(z)$.
The definition of~$\Pi$ implies that~$\Phi$ is a local isometry
and that~$\Phi(\Phi(z)) = z$ for each~$z\in\EE$, so 
$\Phi$ is indeed a (isometric, involutive) diffeomorphism.
For technical reasons, at some points of the proof
we will also need an auxiliary function, which we introduce 
in the following lemma.

\begin{lemma} \label{lemma:Xi}
 There exists a Lipschitz function
 $\Xi\colon\EE\times\EE\to\R$ that satisfies
 the following properties for any~$z_1\in\EE$, $z_2\in\EE$:
 \begin{align}
  \Xi(z_1, \, z_2) = 1 \qquad &\textrm{if and only if }
    z_1 = z_2 \label{Xi} \\
  \Xi(z_1, \, z_2) = -1 \qquad &\textrm{if and only if }
    z_1 = \Phi(z_2) \label{Xibis} \\
  \Xi(z_1, \, z_2) &= \Xi(\Phi(z_1), \, \Phi(z_2)). \label{Xitris}
 \end{align}
\end{lemma}

For instance, when~$\Pi$ is the quotient map~$\S^n\to\R\mathrm{P}^n$,
one such function~$\Xi\colon\S^n\to\S^n$ is simply given
by~$\Xi(z_1, \, z_2) := z_1\cdot z_2$.

\begin{proof}[Proof of Lemma~\ref{lemma:Xi}]
 For a sufficiently small~$\delta_0 > 0$,
 the compact set~$K := \{(z_1, \, z_2)\in\EE\times\EE\colon
 \dist_{\NN}(\Pi(z_1), \, \Pi(z_2))\leq \delta_0\}$
 has two connected components:
 one is~$K_+ := \{(z_1, \, z_2)\in\EE\times\EE
 \colon \dist_{\EE}(z_1, \, z_2) \leq \delta_0\}$,
 the other one is
 $K_- := \{(z_1, \, z_2)\in\EE\times\EE
 \colon \dist_{\EE}(\Phi(z_1), \, z_2) \leq \delta_0\}$.
 (This is a consequence of Remark~\ref{rk:uniform}.)
 Let~$\sigma(z_1, \, z_2) := 1$ for~$(z_1, \, z_2)\in K_+$
 and $\sigma(z_1, \, z_2) := -1$ for~$(z_1, \, z_2)\in K_-$.
 We define
 \[
  \Xi(z_1, \, z_2) := \sigma(z_1, \, z_2) \,
  \left(1 - \frac{1}{\delta_0}\dist_{\EE}(\Pi(z_1), \, \Pi(z_2))\right)_+
 \]
 for all~$z = (z_1, \, z_2)\in\EE\times\EE$,
 where~$t_+ := \max(t, \, 0)$ is the positive part.
 It is not difficult to check that this
 function has all the desired properties.
\end{proof}

\paragraph{Euclidean embeddings and function spaces.}
By Nash isometric embedding, we can identify~$\EE$,
$\NN$ as closed submanifolds of some Euclidean spaces,
say~$\EE\subseteq\R^{\ell}$ and~$\NN\subseteq\R^m$.
Since~$\NN$, $\EE$ are both compact,
the geodesic distance functions~$\dist_{\NN}$, $\dist_{\EE}$
induced by the Riemannian metrics
are equivalent to the Euclidean one.
In other words, there exists a constant~$C$ such that
\begin{gather}
 \abs{w_1 - w_2} \leq \dist_{\NN}(w_1, \, w_2)
  \leq C \abs{w_1 - w_2}, \label{geodistN} \\
 \abs{z_1 - w_2} \leq \dist_{\EE}(z_1, \, z_2)
  \leq C \abs{z_1 - z_2}, \label{geodistE}
\end{gather}
for any~$w_1\in\NN$, $w_2\in\NN$, $z_1\in\EE$, $z_2\in\EE$.
These inequalities can be proved, e.g., reasoning by contradiction,
using the smoothness and compactness of~$\NN$, $\EE$.
Thanks to~\eqref{geodistN}, \eqref{geodistE}, 
we can extend~$\Pi$, $\Phi$, $\Xi$ to Lipschitz
maps~$\Pi\colon\R^{\ell}\to\R^{m}$, $\Phi\colon\R^{\ell}\to\R^{\ell}$,
$\Xi\colon\R^\ell\times\R^\ell\to\R$
(denoted with the same symbols, by abuse of notation).
We define function spaces such as
\[
 \BV(\Omega, \, \NN) := \left\{u\in\BV(\Omega, \, \R^m)\colon u(x)\in\NN
 \textrm{ for a.e. } x\in\Omega\right\} \! ,
\]
which is a metric space with the distance induced by
the norm in~$\BV(\Omega, \, \NN)$.
We define other function spaces, such as 
$\BV(\Omega, \, \EE)$ or~$W^{1,p}(\Omega, \, \NN)$
for instance, in a completely analogous way.
Choosing a different embedding, say, 
$\NN^\prime\subseteq\R^{m^\prime}$
results in a metric space~$\BV(\Omega, \, \NN^\prime)$
that is equivalent to~$\BV(\Omega, \, \NN)$,
in the sense that there exists a Lipschitz map 
$\BV(\Omega, \, \NN)\to\BV(\Omega, \, \NN^\prime)$
with Lipschitz inverse. Indeed, an isometric
diffeomorphism~$\psi\colon \NN\to\NN^\prime$
is also a Lipschitz map with Lipschitz inverse
with respect to the Euclidean distance in
the ambient spaces~$\R^m$, $\R^{m^\prime}$, 
because of~\eqref{geodistN}.
Therefore, composition with~$\psi$ induces 
a Lipschitz map $\BV(\Omega, \, \NN)\to\BV(\Omega, \, \NN^\prime)$
with Lipschitz inverse, by the chain rule 
(see e.g.~\cite[Theorem~3.96]{AmbrosioFuscoPallara}).
It would also be possible to 
define~$\BV(\Omega \, \NN)$, $\BV(\Omega, \, \EE)$ in
a completely intrinsic way, without any reference to a
Euclidean embedding (see for instance the metric space
approach in~\cite{Ambrosio-metricBV}),
but we will not need such generality for our purposes.

\paragraph{Orientable and non-orientable singularities.}
We recall the notion of orientable
and non-orientable singularities of a map~$u\colon\Omega\to\NN$.
Suppose, first, that~$u$ is continuous except for a finite number
of discontinuity points~$p_1, \, \ldots, \, p_n$.
Given a continuous map
$\gamma\colon\S^1\to\Omega\setminus\{p_1, \, \ldots, \, p_n\}$,
we will say that~$u$ is orientable on~$\gamma(\S^1)$
if there exists a continuous map~$v\colon\S^1\to\EE$
such that~$u(\gamma(\omega)) = \Pi(v(\omega))$ for all~$\omega\in\S^1$.
The homotopy lifting property (see~\cite[Proposition~1.30]{Hatcher})
guarantees that, for any closed annulus
$\overline{B}_R(a)\setminus B_r(a)\subseteq
\Omega\setminus\{p_1, \, \ldots, \, p_n\}$,
$u$ is either orientable on both~$\partial B_R(a)$
and~$\partial B_r(a)$, or it is non-orientable on both.
As a consequence, for any singular point~$p_i$,
$u$ is either orientable on~$\partial B_\rho(p_i)$ 
for all~$\rho > 0$ small enough or it is non-orientable 
on~$\partial B_\rho(p_i)$ for all~$\rho > 0$ small enough.
Accordingly, we will say that~$p_i$ is an orientable 
singularity of~$u$ or a non-orientable one.
The homotopy lifting property also implies that orientability
is preserved by uniform convergence: if~$\{u^k\}_{k\in\mathbb{N}}$
is a sequence of smooth maps that converges uniformly on the image
of~$\gamma\colon\S^1\to\Omega\setminus\{p_1, \, \ldots, \, p_n\}$,
then for sufficiently large values of~$k$,
$u^k$ is orientable on~$\gamma(\S^1)$ if and only if
the limit map~$u$ is.

The case
$u\in W^{1,2}_{\mathrm{loc}}(\Omega\setminus\{p_1, \, \ldots, \, p_n\}, \, \NN)$
follows from the previous discussion, 
thanks to the following density result
(see~\cite[Proposition p.~267]{SchoenUhlenbeck2}).
Given a closed set~$K\subseteq\R^2$, we write~$C^\infty(K, \, \NN)$
for the set of functions~$K\to\NN$ that admit a smooth
extension~$U\to\NN$ on some open neighbourhood~$U$ of~$K$.

\begin{prop}[\cite{SchoenUhlenbeck2}] \label{prop:density}
 For any bounded, Lipschitz domain~$D\subseteq\R^2$,
 the set~$C^\infty(\overline{D}, \, \NN)$
 is dense in~$W^{1,2}(D, \, \NN)$.
\end{prop}

Now, if
$u \in W^{1,2}_{\mathrm{loc}}(\Omega\setminus\{p_1, \, \ldots, \, p_n\}, \, \NN)$
then~$u$ is continuous on the circle~$\partial B_\rho(p_i)$
for every~$i$ and almost every value of the radius~$\rho>0$,
by Fubini theorem and Sobolev embedding
$W^{1,2}(\partial B_\rho(p_i))\hookrightarrow C^0(\partial B_\rho(p_i))$.
By Proposition~\ref{prop:density} and a diagonal argument,
there exists a sequence of smooth
functions~$u^k\colon\Omega\setminus\{p_1, \, \ldots, \, p_n\}\to\NN$
that converges to~$u$ in
$W^{1,2}_{\mathrm{loc}}(\Omega\setminus\{p_1, \, \ldots, \, p_n\})$.
Fubini theorem and Sobolev embeddings again give
uniform convergence~$u^k\to u$ on~$\partial B_\rho(p_i)$
for all index~$i$ and a.e.~radius~$\rho$.
As a consequence, $u$ is orientable on~$\partial B_\rho(p_i)$
if and only if~$u^k$, for sufficiently large~$k$, is.
The homotopy lifting property shows that,
for each index~$i$ and almost every~$\rho_1$, $\rho_2$ small enough,
$u$ is orientable on~$\partial B_{\rho_1}(p_i)$
if and only if it is on~$\partial B_{\rho_2}(p_i)$.
Therefore, it still makes sense to say whether~$p_i$ is an orientable
singularity of~$u$ or not.

\paragraph{Sobolev liftings.}

As explained in the introduction, there are several results
available on existence (or non-existence) of liftings 
for Sobolev maps. Nevertheless, we recall here some
of the existing arguments, because they will be useful to us later on.

\begin{lemma} \label{lemma:limit}
 Let~$D\subseteq\R^2$ be a bounded, Lipschitz domain,
 $u\in W^{1,2}(D, \, \NN)$, and~$u^k\in C^\infty(\overline{D},\, \NN)$
 be such that~$u^k\rightharpoonup u$ weakly in~$W^{1,2}(D, \, \NN)$ and pointwise a.e.
 For each~$k\in\mathbb{N}$, let~$v^k\in C^\infty(\overline{D},\, \EE)$
 be a lifting of~$u^k$. Then, we can extract
 a (non-relabelled) subsequence so that
 $v^k\rightharpoonup v$ weakly in~$W^{1,2}(D)$,
 where~$v$ is a lifting of~$u$.
\end{lemma}
\begin{proof}
 By definition, the covering map~$\Pi$
 is a local isometry, i.e.~its differential at any point~$x\in\EE$
 is an isometry between the tangent spaces 
 of~$\EE$ and~$\NN$ at~$x$ and~$\Pi(x)$.
 Since directional derivatives of a map~$w\in W^{1,2}(D, \, \EE)$
 are tangent to~$\EE$ at almost every point, 
 the chain rule implies
 \begin{equation} \label{Lifting-nabla}
  \abs{\nabla w} = \abs{\nabla(\Pi\circ w)} 
  \qquad \textrm{a.e.~in } \Omega \
  \textrm{ for any } w\in W^{1,2}(D, \, \EE).
 \end{equation}
 In particular, $|\nabla u^k| = |\nabla v^k|$
 for all~$k$. Therefore, the sequence~$(v^k)_{k\in\mathbb{N}}$
 is bounded in~$W^{1,2}(D, \, \EE)$ and we can extract a subsequence
 that converges weakly in~$W^{1,2}(D)$,
 strongly in~$L^2(D)$ and almost everywhere to some
 limit~$v\in W^{1,2}(D, \, \R^\ell)$. By taking the limit
 pointwise a.e., we deduce that $v(x)\in\EE$ and~$u(x) = \Pi(v(x))$
 for a.e.~$x\in\Omega$.
\end{proof}

Combining Lemma~\ref{lemma:limit} with Proposition~\ref{prop:density}
and classical existence results for smooth liftings
(see e.g.~\cite[Proposition~1.33]{Hatcher}), 
we obtain existence results for $W^{1,2}$-liftings. 
For example, any map~$u\in W^{1,2}(D, \, \NN)$ in a bounded,
Lipschitz, simply connected domain~$D\subseteq\R^2$ admits a
lifting~$v\in W^{1,2}(D, \, \EE)$. This fact remains true
in higher-dimensional simply connected domains~$D\subseteq\R^n$,
but the proof is more delicate; see e.g.~\cite{BethuelChiron}.
We also have uniqueness of the lifting in~$W^{1,2}$,
up to composition with~$\Phi$.

\begin{lemma} \label{lemma:uniqueness}
 Let~$D\subseteq\R^2$ be a bounded, Lipschitz domain.
 Let~$v_1\in W^{1,2}(D, \, \EE)$, $v_2\in W^{1,2}(D, \, \EE)$
 be liftings of the same map~$u\in W^{1,2}(D, \, \EE)$.
 Then, either~$v_1 = v_2$ a.e.~in~$\Omega$ or~$v_1 = \Phi(v_2)$ a.e.~in~$\Omega$.
\end{lemma}
\begin{proof}
 Let~$w := \Xi(v_1, \, v_2)$, where~$\Xi$ is the function
 given by Lemma~\ref{lemma:Xi}.
 Since~$\Xi$ is Lipschitz-continuous, the chain rule implies
 that~$w\in W^{1,2}(\Omega, \, \R)$. Moreover,
 for almost every~$x\in\Omega$ we have $\Pi(v_1(x)) = \Pi(v_2(x))$,
 and hence $w(x)\in \{1, \, -1\}$, by~\eqref{Xi}, \eqref{Xibis}.
 Since~$\{1, \, -1\}$ is a discrete set,
 it follows that~$w$ must be constant.
 The lemma follows, again by~\eqref{Xi}, \eqref{Xibis}.
\end{proof}
}

\subsection{Properties of minimal relative connections}
\label{sect:minrelconn}

{We collect here a couple of observations
on minimal connections which will be useful later on.}

\begin{lemma} \label{lemma:minbd}
 Let~$\Omega\subseteq\R^2$ be a bounded, simply connected domain
 of class~$C^1$. Let $a_1$, \ldots $a_p$ be distinct points in~$\Omega$, and
 let~$\{L_1, \, \ldots, \, L_q\}$ be a minimal connection
 for~$\{a_1, \, \ldots, \, a_p\}$ relative to~$\Omega$.
 If~$L_j$ has an endpoint~$b_j\in\partial\Omega$,
 then~$L_j$ is orthogonal to~$\partial\Omega$ at~$b_j$.
\end{lemma}
\begin{proof}
 Let~$a_j$ be the endpoint of~$L_j$ that is not
 on~$\partial\Omega$. By minimality of~$\{L_1, \, \ldots, \, L_q\}$,
 $b_j$ must be a minimiser of the
 function~$f(x) := \abs{x - a_j}^2$
 subject to the constraint~$x\in\partial\Omega$.
 Then, the lemma follows easily
 from Lagrange multipliers theorem.
\end{proof}

\begin{figure}[t]
 \centering
 \begin{subfigure}{.4\textwidth}
\begin{tikzpicture}[scale=1.1]
\definecolor{lightblue}{RGB}{150,180,255}

\def\DomainPath{
  plot[domain=-2:2,samples=400,variable=\y]
  ({-1 + 2^(-3.7*\y*\y + 1.2)}, {\y})
}

\draw[fill=lightblue] \DomainPath -- (4.5, 2) -- (4.5, -2) -- cycle;

\coordinate (b1) at (0.7,1);
\coordinate (c2) at (3,1.3);
\coordinate (b2) at (0.7,-1);
\coordinate (c1) at (3.5,-0.7);
\coordinate (d1) at (0.7, 0.342);
\coordinate (d2) at (0.7, -0.342);

\fill (b1) circle (1.2pt);
\fill (c1) circle (1.2pt);
\fill (b2) circle (1.2pt);
\fill (c2) circle (1.2pt);
\fill (d1) circle (1.2pt);
\fill (d2) circle (1.2pt);

\node[left]  at ($(b1) + (0.1, 0.1)$) {$b_1$};
\node[right] at (c1) {$c_1$};
\node[left]  at ($(b2) + (0.15, -0.15)$) {$b_2$};
\node[right] at (c2) {$c_2$};
\node[left]  at ($(d1) + (0.1, 0.1)$) {{\color{red} $d_1$}};
\node[left]  at ($(d2) + (0.1, 0.1)$) {{\color{red} $d_2$}};

\draw[thick, name path=1] (b1) -- (c1);
\draw[thick, name path=2] (b2) -- (c2);
\path [name intersections={of=1 and 2,by=o}];
\node[above] at (o) {$x_0$};
\fill (o) circle (1.2pt);

\draw[red, thick, dashed] (b2) -- (d2);
\draw[red, thick, dashed] (b1) -- (d1);
\draw[red, thick, dashed] (c1) -- (c2);

\end{tikzpicture}

  \caption{}
  \label{fig:disjoint1}
 \end{subfigure} \ 
 \begin{subfigure}{.33\textwidth}
\begin{tikzpicture}[scale=1.1]
\definecolor{lightblue}{RGB}{150,180,255}

\def\DomainPath{
  plot[domain=-1:3.5, samples=400, variable=\t]
  ({\t}, {-2 + 2^(-3.7*(\t - 1.5)*(\t  - 1.5) + 0.7)})
}

\draw[fill=lightblue] \DomainPath -- (3.5, 2) -- (-1, 2) --  cycle;

\coordinate (b1) at (0.5,1);
\coordinate (b2) at (0,-1);
\coordinate (b3) at (2.7,-1);
\coordinate (c1) at (1.7,0.4);
\coordinate (d2) at (1.065,-1);
\coordinate (d3) at (1.935,-1);

\fill (b1) circle (1.2pt);
\fill (b2) circle (1.2pt);
\fill (b3) circle (1.2pt);
\fill (c1) circle (1.2pt);
\fill (d2) circle (1.2pt);
\fill (d3) circle (1.2pt);

\node[left]  at ($(b1) + (0.1, 0.1)$) {$b_1$};
\node[left]  at ($(b2) + (0.15, -0.15)$) {$b_2$};
\node[right]  at ($(b3) + (0, -0.15)$) {$b_3$};
\node[right] at (c1) {$c_1 = c_2$};
\node[right]  at ($(d2) + (-0.25, -0.3)$) {{\color{red} $d_2$}};
\node[left]  at ($(d3) + (0.3, -0.3)$) {{\color{red} $d_3$}};

\draw[red, thick] (b1) -- (c1);
\draw[thick] (b2) -- (c1);
\draw[thick] (b3) -- (c1);

\draw[red, thick, dashed] (b2) -- (d2);
\draw[red, thick, dashed] (b3) -- (d3);

\end{tikzpicture}

  \caption{}
  \label{fig:disjoint2}
 \end{subfigure} \ 
 \begin{subfigure}{.24\textwidth}
  \begin{tikzpicture}[scale=1.1]
\definecolor{lightblue}{RGB}{150,180,255}

\def\DomainPath{
  plot[domain=-1.3:1.3, samples=400, variable=\t]
  ({\t}, {-2 + 2^(-3*\t*\t + 0.3)})
}

\draw[fill=lightblue] \DomainPath -- (1.3, 2) -- (-1.3, 2) --  cycle;

\coordinate (b1) at (0,1);
\coordinate (b2) at (0,0);
\coordinate (c1) at (0,-0.768856);

\fill (b1) circle (1.2pt);
\fill (b2) circle (1.2pt);
\fill (c1) circle (1.2pt);

\node[right] at (b1) {$b_1$};
\node[right] at (b2) {$b_2$};
\node[right] at (c1) {$c_1 = c_2$};

\draw[thick] (b2) -- (c1);
\draw[red, thick] (b1) -- (b2);

\end{tikzpicture}

  \caption{}
  \label{fig:disjoint3}
 \end{subfigure}
  \caption{Illustration of the proof of Lemma~\ref{lemma:mindisj}.}
  \label{fig:disjoint}
 \end{figure}
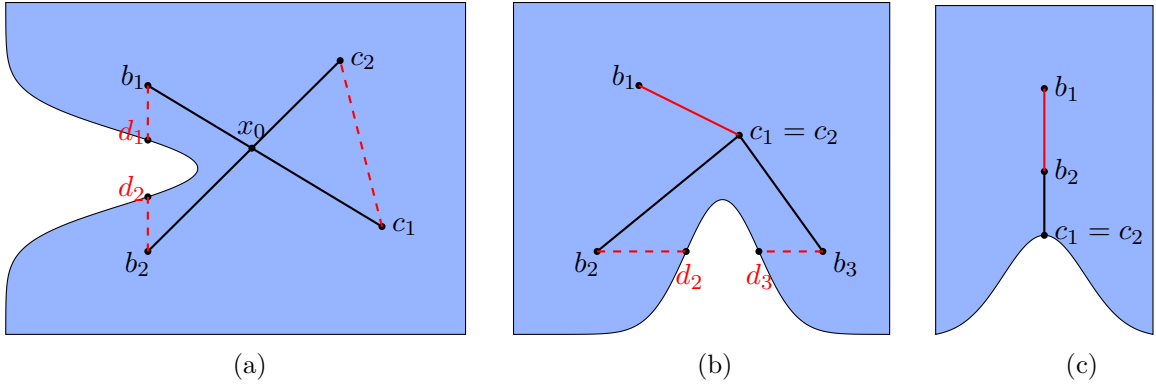

\begin{lemma} \label{lemma:mindisj}
 Let~$\Omega\subseteq\R^2$ be a bounded, simply connected domain
 of class~$C^1$. Let $a_1$, \ldots $a_p$ be distinct points in~$\Omega$, and
 let~$\{L_1, \, \ldots, \, L_q\}$ be a minimal connection
 for~$\{a_1, \, \ldots, \, a_p\}$ relative to~$\Omega$.
 Then, the line segments~$L_j$ are pairwise disjoint.
 Moreover, for each index~$i$ such that~$a_i\in\Omega$,
 there is exactly one index~$j$ such that~$a_i$
 is an endpoint of~$L_j$. Finally, for each~$j$,
 the intersection~$L_j\cap\partial\Omega$
 is either empty or an endpoint of~$L_j$.
\end{lemma}
\begin{proof}
 Throughout this proof, we will denote by~$[a, \, b]$
 the straight line segment of endpoints~$a\in\R^2$, $b\in\R^2$.
 Suppose, towards a contradiction, that
 $L_1$, \ldots, $L_q$ are not pairwise disjoint.
 Say, for instance, that there exists~$x_0\in L_1\cap L_2$,
 where~$L_1 = [b_1, \, c_1]$, $L_2 = [b_2, \, c_2]$ are distinct.
 Suppose first that~$b_1$, $c_1$, $b_2$, $c_2$
 are distinct (this is the case illustrated 
 in Figure~\ref{fig:disjoint1}). 
 Then, the triangle inequality implies
 \[
  \abs{b_1 - b_2} + \abs{c_1 - c_2}
  < \abs{b_1 - x_0} + \abs{b_2 - x_0}
   + \abs{c_1 - x_0} + \abs{c_2 + x_0}
  = \abs{b_1 - c_1} + \abs{b_2 - c_2}
 \]
 If~$[b_1, \, b_2]$, $[c_1, \, c_2]$
 are contained in~$\overline{\Omega}$,
 then the family $\{L_3, \, L_4, \, \ldots, L_q, \,
 \, [b_1, \, b_2], \, [c_1, \, c_2]\}$
 is a connection for~$\{a_1, \, \ldots, \, a_p\}$
 relative to~$\Omega$ with smaller total length
 than~$\{L_1, \, \ldots, \, L_p\}$, which is impossible
 because~$\{L_1, \, \ldots, \, L_p\}$ is minimal.
 Therefore, at least one between~$[b_1, \, b_2]$ and~$[c_1, \, c_2]$
 must contain a point of~$\R^2\setminus\overline{\Omega}$.
 Say, for instance, $[b_1, \, b_2]\not\subseteq\overline{\Omega}$
 but~$[c_1, \, c_2]\subseteq\overline{\Omega}$.
 Let~$d_1$, $d_2$ be the points of~$[b_1, \, b_2]\cap\partial\Omega$
 closest to~$b_1$, $b_2$ respectively.
 These points exists, because~$[b_1, \, b_2]\cap\partial\Omega$
 is compact, and are unique. By definition, $[b_1, \, d_1]$
 and~$[b_2, \, d_2]$ are contained in~$\overline{\Omega}$,
 so the family $\{L_3, \, L_4, \, \ldots, \, L_q, \,
 \, [b_1, \, d_1], \, [b_2, \, d_2], \, [c_1, \, c_2]\}$
 is a connection for~$\{a_1, \, \ldots, \, a_p\}$
 relative to~$\Omega$ and its total length
 is smaller than that of~$\{L_1, \, \ldots, \, L_p\}$.
 This is a contradiction.
 In case neither~$[b_1, \, b_2]$
 nor~$[c_1, \, c_2]$ is contained in~$\overline{\Omega}$,
 we obtain a contradiction in a similar way.

 It remains to consider the case the points~$b_1$, $c_1$, $b_2$, $c_2$
 are \emph{not} distinct, which means, $L_1$ and~$L_2$
 meet at a common endpoint. Say, for instance, $c_1 = c_2 = x_0$
 but~$b_1\neq b_2$ (otherwise, $L_1$ and~$L_2$ would coincide).
 Suppose~$x_0$ is one of the points~$a_1$, \ldots, $a_p$
 {(see Figure~\ref{fig:disjoint2})}.
 Since the number of line segments in the connection
 that have an endpoint at~$x_0$ is odd, by definition,
 there is at least another segment~$L_j$
 that has an endpoint at~$x_0$, say~$L_3 = [b_3, \, x_0]$.
 If~$[b_2, \, b_3]\subseteq\overline{\Omega}$,
 then 
 $\{L_1, \, L_4, \, L_5, \, \ldots, \, L_q,
 \, [b_1, \, c_1], \, [b_2, \, b_3] \}$
 is a connection for~$a_1, \, \ldots, \, a_p$ relative to~$\Omega$
 with smaller total length then~$\{L_1, \, \ldots, \, L_q\}$,
 which is a contradiction.
 Otherwise, we obtain a contradiction by replacing~$[b_2, \, b_3]$
 with smaller line segments $[b_2,\,d_2]$, $[b_3,\,d_3]$ 
 contained in~$\overline{\Omega}$, as above.
 If~$x_0$ is a point in~$\partial\Omega$ 
 {(as in Figure~\ref{fig:disjoint3})}, then both~$L_1$ and~$L_2$
 must be orthogonal to~$\partial\Omega$ at~$x_0$, by Lemma~\ref{lemma:minbd}.
 Since both~$L_1$ and~$L_2$ must be contained in~$\overline{\Omega}$,
 which is locally the subgraph of a regular function,
 one between~$L_1$ and~$L_2$ must
 be contained in the other one. Then, we obtain
 a contradiction by considering 
 $\{[b_1, \, c_1], \, L_3, \, L_4, \, \ldots, \, L_q\}$,
 which is a connection for~$\{a_1, \, \ldots, \, a_p\}$
 relative to~$\Omega$ and violates the minimality
 of $\{L_1, \, L_2, \, \ldots, \, L_q\}$.
 This proves that the line segments~$L_1, \, \ldots, \, L_p$
 are pairwise disjoint.

 By definition of connection relative to~$\Omega$,
 for each~$i$ with~$a_i\in\Omega$ there must be
 an odd number of indices~$j$
 --- hence, at least one --- such that~$a_i$ belongs to~$L_j$.
 Since the~$L_j$ are pairwise disjoint, such~$j$ is unique.
 Finally, if~$L_j$ is a line segment that intersects~$\partial\Omega$
 at more than one point, we construct
 a connection for~$\{a_1, \, \ldots, \, a_p\}$ relative to~$\Omega$
 with smaller total length than~$\{L_1, \, \ldots, \, L_q\}$
 by replacing~$L_j$ with smaller line segments, as above.
\end{proof}

\subsection{Liftings and sets of finite perimeter}
\label{sect:lifting-reformulate}

Given a set~$E\subseteq\R^2$, we will denote by~$\partial^* E$
the essential boundary of~$E$, which is defined as
\[
 \partial^* E:=\left\{x\in\R^2\colon
 \limsup_{\rho\to 0}\frac{\abs{B_\rho(x)\cap E}}{\rho^2} > 0, \
 \limsup_{\rho\to 0}\frac{\abs{B_\rho(x)\setminus E}}{\rho^2} > 0
 \right\} \! .
\]
Given a measurable subset~$E\subseteq\Omega$,
we define its perimeter relative to~$\Omega$ as
\[
 P(E, \, \Omega)
 := \sup\left\{\int_\Omega \div\mathbf{V}(x) \, \d x \colon
 \mathbf{V}\in C^\infty_{\mathrm{c}}(\Omega, \, \R^2),
 \ \norm{\mathbf{V}}_{L^\infty(\Omega)} \leq 1 \right\}
 < +\infty.
\]
For smooth sets, $\partial^* E = \partial E$, and the perimeter~$P(E, \, \Omega)$
coincides with the surface area of~$\Omega\cap\partial E$,
by the Gauss-Green theorem.
More generally, a result of Federer~\cite[4.5.11]{Federer}
implies that~$P(E, \, \Omega) = \H^1(\Omega\cap \partial^* E)$.
If the set~$\Omega$ has smooth, or even just Lipschitz,
boundary and~$E$ has finite perimeter relative to~$\Omega$,
then~$E$ has finite perimeter in~$\R^2$.
(This follows from the definition of perimeter,
via a localisation argument based on cut-off functions
that approximate the indicator function of~$\Omega$.)
Since in the sequel we will be only be interested
in subsets of a regular domain~$\Omega$,
we will simply write that a set `has finite perimeter'
and will not specify whether the perimeter
is taken relative to~$\Omega$ or~$\R^2$.

Given two sets~$A$, $B$, we will denote by~$A\triangle B
:= (A\setminus B)\cup (B\setminus A)$ their symmetric difference.
{The symmetric difference is compatible with
the essential boundary, as made precise
by the following lemma.

\begin{lemma} \label{lemma:boundary}
 If~$A$, $B$ are sets of finite perimeter,
 then~$A\triangle B$ has finite perimeter and
 \begin{equation} \label{boundary}
  \partial^*(A\triangle B)
  = \partial^* A \triangle \partial^*B \mod\H^1.
 \end{equation}
\end{lemma}
In this proof and later on, we will denote
by~$\J_{w}$ the jump set of a function~$w\in\BV(\Omega)$
and by~$w^+$, $w^-$ the traces of~$w$ on either side of the jump set.
(See e.g.~\cite[Chapter~3]{AmbrosioFuscoPallara} for more details.)
We extend~$w^+$, $w^-$ to functions~$\widetilde{w}^+$, $\widetilde{w}^-$
defined~$\H^1$-almost everywhere by setting~$\widetilde{w}^+(x) := w^+(x)$,
$\widetilde{w}^-(x) := w^-(x)$ for~$x\in\J_u$ and
\[
 \widetilde{w}^+(x) = \widetilde{w}^-(x) :=
 \lim_{\rho\to 0} \frac{1}{\abs{B_\rho(x)}}\int_{B_\rho(x)} w(y) \, \d y
\]
otherwise. By Federer-Vol'pert theorem
(see e.g.~\cite[Theorem~3.78]{AmbrosioFuscoPallara}),
this limit exists in~$\R$ for~$\H^1$-a.e.~$x\in\Omega\setminus\J_w$.

\begin{proof}[Proof of Lemma~\ref{lemma:boundary}]
 This lemma is essentially a consequence of
 the chain rule for BV-functions
 (see e.g.~\cite[Theorem~3.96]{AmbrosioFuscoPallara}).
 Indeed, let~$\chi_A\in\BV(\R^2)$, $\chi_B\in\BV(\R^2)$
 be the indicator functions of~$A$, $B$
 respectively (i.e., $\chi_A := 1$ in~$A$,
 $\chi_A := 0$ in~$\R^2\setminus A$, and similarly for~$\chi_B$).
 The jump sets~$\J_{\chi_A}$, $\J_{\chi_B}$ are equal to~$\partial^*A$,
 $\partial^* B$ respectively, modulo~$\H^1$-null sets
 (see e.g.~\cite[Example~3.68]{AmbrosioFuscoPallara}).
 By the chain rule, the function
 \[
  \chi_{A\triangle B} = \max(\chi_A, \, \chi_B) - \min(\chi_A, \, \chi_B)
 \]
 belongs to~$\BV(\R^2)$ and its jump set
 coincides, modulo~$\H^1$-null sets,
 with the set of points~$x\in\partial^*A\cup\partial^*B$
 such that
 \begin{equation*}
  \begin{split}
   &\max(\widetilde{\chi_A}^+(x), \, \widetilde{\chi_B}^+(x))
   - \min(\widetilde{\chi_A}^+(x), \, \widetilde{\chi_B}^+(x)) \\
   &\hspace{2cm} \neq
   \max(\widetilde{\chi_A}^-(x), \, \widetilde{\chi_B}^-(x))
   - \min(\widetilde{\chi_A}^-(x), \, \widetilde{\chi_B}^-(x)).
  \end{split}
 \end{equation*}
 By direct inspection, we see this condition
 is satisfied if and only if~$x\in \partial^*A \triangle \partial^*B$,
 except possibly for an~$\H^1$-negligible set of~$x$'s.
\end{proof}

Lemma~\ref{lemma:boundary} suggests the presence of
an underlying homological structure. Indeed, the symmetric
difference~$\triangle$ induces Abelian group structures
on the classes of finite perimeter sets in~$\R^2$ and
$1$-rectifiable subsets of~$\R^2$, modulo $\H^1$-null sets.
The equality~\eqref{boundary} shows that the essential
boundary~$\partial^*$ is a group homomorphism.
All that is missing from a completely formed homology theory
is a suitable notion of boundary for~$1$-rectifiable sets in~$\R^2$.
As it turns out, this is precisely the notion
of boundary in the sense of currents modulo two
(or equivalently, flat chains with coefficients in~$\Z/2\Z$).}

Throughout the rest of Section~\ref{sect:lifting}, we fix
distinct points~$a_1$, \ldots, $a_p$, $b_1$, \ldots, $b_r$
in~$\Omega$ and a map
\[
 u\in W^{1,1}(\Omega, \, \NN)\cap
  W^{1,2}_{\loc}(\Omega\setminus
  \{a_1, \, \ldots, a_p, \, b_1, \, \ldots, \, b_r\}, \, \NN)
\]
with non-orientable singularities at the points~$a_j$
and orientable ones at the points~$b_h$.
We also fix a minimal connection~$\{L_1, \, \ldots, \, L_q\}$
for~$\{a_1, \, \ldots, \, a_p\}$ relative to~$\Omega$.
We start by constructing a lifting~$v^\star$ of~$u$
whose jump set is the union of all line segments~$L_j$.

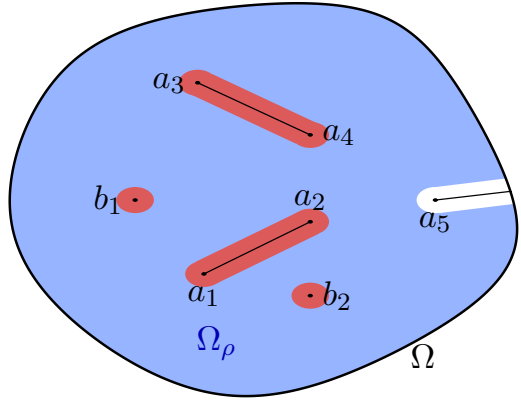
\begin{figure}[t]
  \centering
  \resizebox{0.35\textheight}{!}{%
  \begin{tikzpicture}[scale=0.7, yscale=0.7]
\definecolor{lightblue}{RGB}{150,180,255}
\colorlet{textblue}{blue!70!black}
\colorlet{lightred}{red!40}
\definecolor{myred}{RGB}{220,90,90}



\def\DomainPath{
  plot[domain=0:360,samples=400,smooth,variable=\t]
  ({4.2 + (4 - 0.2*sin(2*\t) - 0.2*sin(3*\t))*cos(\t)},
   {1.2 - (4.5 - 0.2*sin(-2*\t) - 0.2*sin(-3*\t))*sin(-\t)})
}

\fill[lightblue] \DomainPath -- cycle;

\node at (3.5,-2.3) {{\color{textblue}$\Omega_\rho$}};
\node at (6.8,-2.6) {$\Omega$};


\coordinate (b1) at (2.2,1);
\coordinate (b2) at (5,-1.2);

\fill[myred] (b1) circle (0.3);
\fill[myred] (b2) circle (0.3);

\fill (b1) circle (1.2pt);
\fill (b2) circle (1.2pt);

\node[left]  at (b1) {$b_1$};
\node[right] at (b2) {$b_2$};


\coordinate (a1) at (5,0.5);
\coordinate (a2) at (3.3,-0.7);

\begin{scope}[rotate around={35:(4.15,-0.1)}]
 \fill[myred]
  ($(4.15,-0.1)+(-1.0,-0.3)$)
  rectangle
  ($(4.15,-0.1)+(1.0,0.3)$);
\end{scope}

\fill[myred] (a1) circle (0.3);
\fill[myred] (a2) circle (0.3);
\draw (a1)--(a2);
\fill (a1) circle (1.2pt);
\fill (a2) circle (1.2pt);
\node[above] at (a1) {$a_2$};
\node[below] at (a2) {$a_1$};


\coordinate (a3) at (3.2, 3.7);
\coordinate (a4) at (5, 2.5);

\begin{scope}[rotate around={-34:(4.1,3.1)}]
 \fill[myred]
  ($(4.1,3.1)+(-1.0,-0.29)$)
  rectangle
  ($(4.1,3.1)+(1.0,0.29)$);
\end{scope}

\fill[myred] (a3) circle (0.3);
\fill[myred] (a4) circle (0.3);
\draw (a3)--(a4);
\fill (a3) circle (1.2pt);
\fill (a4) circle (1.2pt);
\node[left]  at (a3) {$a_3$};
\node[right] at (a4) {$a_4$};


\coordinate (a5) at (7,1);
\coordinate (c)  at (8.2,1.2);

\begin{scope}[rotate around={10:(7.6,1.1)}]

\fill[white]
  ($(7.6,1.1)+(-0.6,-0.3)$)
  rectangle
  ($(7.6,1.1)+(1.2,0.3)$);
\end{scope}

\fill[white] (a5) circle (0.3);
\draw (a5)--(c);
\fill (a5) circle (1.2pt);

\node[below] at (a5) {$a_5$};

\draw[thick] \DomainPath -- cycle;

\end{tikzpicture}

  }
  \caption{The set~$\Omega_\rho$ introduced in the proof
  of Lemma~\ref{lemma:goodlifting} (in blue). The set~$G_\rho$
  is the union of~$\Omega_\rho$ and the red regions.}
  \label{fig:Grho}
 \end{figure}

\begin{lemma} \label{lemma:goodlifting}
 Given~$u$ and~$L_1$, \ldots $L_q$ as above,
 there exists a lifting~$v^\star\in\SBV(\Omega, \, \EE)$
 of~$u$ such that
 \begin{equation} \label{goodlifting}
  \J_{v^\star} = \bigcup_{j=1}^q L_j \qquad \mod\H^1.
 \end{equation}
\end{lemma}
\begin{proof}
 We proceed along the lines of~\cite[Lemma~A.3]{CanevariMajumdarStroffoliniWang}.
 For any~$\rho>0$ and~$j\in\{1, \, \ldots, \, d\}$, we define
 \[
  U_{j,\rho} := \left\{x\in\R^2\colon
  \dist(x, \, L_j) < \rho \right\} \! .
 \]
 Let~$J$ be the set of indices~$j\in\{1, \, \ldots, \, q\}$
 such that~$L_j$ has an endpoint in~$\partial\Omega$, let
 \[
  G_\rho := \left\{x\in\Omega \colon \dist(x, \, \partial\Omega) > \rho \right\}
  \setminus \bigcup_{j\in J} U_{j,\rho}
 \]
 and
 \[
  \Omega_\rho := G_\rho\setminus \left(\bigcup_{j\notin J} U_{j,\rho}
  \cup \bigcup_{h = 1}^r B_\rho(b_h)\right)\! .
 \]
 {The sets~$G_\rho$, $\Omega_\rho$ are illustrated in Figure~\ref{fig:Grho}.}
 For small enough values of~$\rho$,
 the set~$G_\rho$ is a simply connected domain with
 Lipschitz boundary. Moreover, $G_\rho$ contains an \emph{even}
 number of non-orientable singularities of~$u$.
 Indeed, the non-orientable singularities~$a_i$ inside~$G_\rho$
 are exactly the endpoints of~$L_j$ with~$j\notin J$.
 Since the~$L_j$'s are mutually disjoint (by Lemma~\ref{lemma:mindisj}),
 there must be an even number of such endpoints.
 In particular, for~$\rho$ small enough, $u$ is orientable
 on each boundary component of~$\Omega_\rho$ and of class~$W^{1,2}$
 inside~$\Omega_\rho$.

 {We claim that for a.e.~$\rho$
 there exists a lifting~$v^\star_{\rho}\in W^{1,2}(\Omega_\rho, \, \EE)$
 of~$u_{|\Omega_\rho}$. Indeed, by Proposition~\ref{prop:density}
 and a diagonal argument, we find a sequence of smooth maps
 \[
  u^k\colon\Omega\setminus\{a_1, \, \ldots, \, a_p, \, b_1, \, \ldots, \, b_h\}\to\NN
 \]
 that converges to~$u$ in $W^{1,2}_{\mathrm{loc}}(\Omega
 \setminus\{a_1, \, \ldots, \, a_p, \, b_1, \, \ldots, \, b_h\})$.
 By Fubini theorem and the Sobolev embedding~$W^{1,2}(\partial\Omega_\rho)\hookrightarrow C^0(\partial\Omega_\rho)$, we also have uniform convergence
 $u^k\to u$ on~$\partial\Omega_\rho$ for a.e.~$\rho$.
 Therefore, for sufficiently large values of~$k$,
 $u^k$ is orientable on each component of~$\partial\Omega_\rho$.
 By classical lifting results for smooth maps
 (see e.g.~\cite[Proposition~1.33]{Hatcher}),
 it follows that $u^k|_{\overline{\Omega_\rho}}$ has a smooth lifting~$v^k_\rho\colon\overline{\Omega_\rho}\to\NN$.
 Lemma~\ref{lemma:limit} allows us to extract a
 (non-relabelled) subsequence such that~$v^k_\rho \rightharpoonup v^\star_\rho$,
 where~$v^\star_\rho\in W^{1,2}(\Omega_\rho\, \, \EE)$
 is a lifting of~$u_{\Omega_\rho}$, as claimed.
 Moreover, for~$0 < \rho_1 < \rho_2$ we have
 either~$v^\star_{\rho_2} = v^\star_{\rho_1}$ a.e.~in~$\Omega_{\rho_2}$
 or~$v^\star_{\rho_2} = \Phi(v^\star_{\rho_1})$ a.e.~in~$\Omega_{\rho_2}$,
 because of Lemma~\ref{lemma:uniqueness}.}
 Therefore, we can choose a decreasing sequence~$\rho_k\searrow 0$
 and make sure that $v^\star_{\rho_{k+1}} = v^\star_{\rho_k}$
 a.e.~in~$\Omega_{\rho_{k+1}}$
 for any~$k$, upon replacing~$v^\star_{\rho_k}$ with~$\Phi(v^\star_{\rho_k})$
 if necessary. By glueing, we define a lifting
 \[
  v^\star\in W^{1,2}_{\loc}\left(\Omega\setminus
  \left(\bigcup_{j=1}^q L_j\cup\{b_1, \, \ldots,
  \, b_r\}\right)\!, \, \EE\right)
 \]
 of~$u$. In fact, we have
 $v^\star\in W^{1,1}(\Omega\setminus\cup_{j=1}^q L_j, \, \EE)$
 because
 $\abs{\nabla v^\star} = \abs{\nabla u}$ a.e. (see~\eqref{Lifting-nabla})
 and $\nabla u$ is integrable in~$\Omega$ by assumption.
 We also have~$v^\star\in\SBV(\Omega, \, \EE)$,
 because~$\bigcup_j L_j$ has finite length and~$v^\star$ is bounded
 (see~\cite[Proposition~4.4]{AmbrosioFuscoPallara}).

 It only remains to prove~\eqref{goodlifting}.
 By construction, we have $\J_{v^\star}\subseteq\bigcup_j L_j$;
 we claim that~$\J_{v^\star}$ contains $\H^1$-almost all of~$\bigcup_j L_j$.
 The crucial observation here is that, for any~$j\in\{1, \, \ldots, \, q\}$,
 at least an endpoint of~$L_j$ is a non-orientable singularity of~$u$.
 (This is immediate from the definition of relative connection.)
 Then, by applying Fubini theorem,
 for almost every~$x\in L_j$
 we can construct a rectangle~$K_x\csubset\Omega$
 that contains exactly one non-orientable singularity of~$u$
 (one of the endpoints of~$L_j$)
 and is such that~$\partial K_x\cap (\bigcup_j L_j) = \{x\}$,
 $u\in W^{1,2}(\partial K_x, \, \NN)$, $v^\star\in
 W^{1,2}_{\mathrm{loc}}(\partial K_x\setminus\{x\}, \, \EE)$.
 On the boundary of such a rectangle, $u$ is continuous
 but non-orientable, i.e.~it has no continuous lifting.
 But since~$v^\star$ is continuous in~$\partial K_x\setminus\{x\}$,
 we must have~$x\in\J_{v^\star}$. Therefore,
 $\J_{v^\star}$ contains $\H^1$-almost all of~$\bigcup_j L_j$.
\end{proof}

For any set of finite perimeter~$A\subseteq\Omega$,
we define
\begin{equation} \label{L_A}
 \L_A := (\Omega\cap\partial^* A)
  \triangle\bigcup_{j=1}^q L_j.
\end{equation}
An example is given in Figure~\ref{fig:L_A}.
We recall that the minimal connection~$\{L_1, \, \ldots, \, L_p\}$
has been chosen once and for all.

\begin{figure}[t]
  \centering
 \resizebox{.35\textheight}{!}{%
 \begin{tikzpicture}[scale=0.7, yscale=0.7]
\definecolor{lightblue}{RGB}{150,180,255}
\colorlet{textblue}{lightblue!35!blue!60!black}
\definecolor{myred}{RGB}{220,90,90}
\colorlet{darkred}{myred!50!black}


\def\DomainPath{
  plot[domain=0:360,samples=400,smooth,variable=\t]
  ({4.2 + (4 - 0.2*sin(2*\t) - 0.2*sin(3*\t))*cos(\t)},
   {1.2 - (4.5 - 0.2*sin(-2*\t) - 0.2*sin(-3*\t))*sin(-\t)})
}

\draw[thick, fill=lightblue] \DomainPath -- cycle;

\coordinate (a1) at (2,3);
\coordinate (a2) at (5.5,3.2);
\coordinate (a3) at (2.2,-0.2);
\coordinate (a4) at (5.2,-1.2);

\coordinate (b1) at ($(a1)!0.25!(a2)$);
\coordinate (b2) at ($(a1)!0.8!(a2)$);
\coordinate (c1) at ($(a3)!0.22!(a4)$);
\coordinate (c2) at ($(a3)!0.73!(a4)$);
\coordinate (o) at ($(b2)!0.6!(c1) - (0, 0.4)$);

\fill (a1) circle (1.2pt);
\fill (a2) circle (1.2pt);
\fill (a3) circle (1.2pt);
\fill (a4) circle (1.2pt);

\node[left]  at (a1) {$a_1$};
\node[right] at (a2) {$a_2$};
\node[left]  at (a3) {$a_3$};
\node[right] at (a4) {$a_4$};

\node[draw=darkred, thick, fill=myred, circle through={(b2)}] at ($(b1)!0.75!(b2)$) {};
\node[draw=darkred, thick, fill=myred, circle through={(c2)}] at ($(c1)!0.5!(c2)$) {};
\fill[myred] (b1) -- (b2) -- (c2) -- (c1) -- cycle;
\node[draw=darkred, thick, fill=lightblue, circle, minimum size = 0.7] at (o) {};

\draw[darkred, thick] (b2) -- ($(b1)!0.5!(b2)$);
\draw[darkred, thick] (c1) -- (c2);
\draw[textblue, thick] (a2) -- (b2);
\draw[textblue, thick] (b2) -- (c2);
\draw[textblue, thick] (c2) -- (a4);
\draw[textblue, thick] (a1) -- (b1);
\draw[textblue, thick] (b1) -- (c1);
\draw[textblue, thick] (c1) -- (a3);

\node at (6.8,-2.8) {$\Omega$};
\node at ($(b1)!0.5!(b2) - (0, 1)$) {{\color{darkred}$A$}};


\end{tikzpicture}
}
  \caption{The set~$\L_A$ defined in~\eqref{L_A}. In this example, $A$ is the set in red and $\L_A$ consists of two polygonal paths (in blue) joining~$a_1$ with~$a_3$ and~$a_2$ with~$a_4$, as well as a finite union of loops.}
  \label{fig:L_A}
 \end{figure}
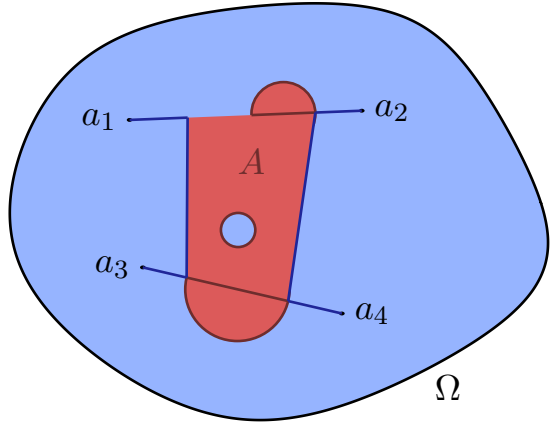

\begin{lemma} \label{lemma:liftbd}
 Let~$\{L_1, \, \ldots, \, L_q\}$ be a minimal connection
 for~$\{a_1, \, \ldots, \, a_p\}$ relative to~$\Omega$.
 For any lifting~$v\in\SBV(\Omega, \, \EE)$ of~$u$,
 there exists a set~$A\subseteq\Omega$ of finite
 perimeter such that
 \begin{equation} \label{genlifting}
  \J_{v} = \L_A \qquad \mod\H^1.
 \end{equation}
 Conversely, given any set~$A\subseteq\Omega$ of finite perimeter
 there exists a lifting~$v\in\SBV(\Omega, \, \EE)$
 that satisfies~\eqref{genlifting}.
\end{lemma}

{As a preliminary remark to the proof,
we observe that for any~$v\in\SBV(\Omega, \, \EE)$
and $\H^1$-a.e.~$x\in\J_v$, there holds
$v^+(x) \in\EE$, $v^-(x)\in\EE$.
Indeed, the traces~$v^+(x)$, $v^-(x)$ satisfy
\[
 \lim_{\rho\to 0} \frac{1}{\abs{B^+_\rho(x)}}
  \int_{B^+_\rho(x)} \abs{v(y) - v^+(x)} \d y = 0,
 \qquad \lim_{\rho\to 0} \frac{1}{\abs{B^-_\rho(x)}}
  \int_{B^-_\rho(x)} \abs{v(y) - v^-(x)} \d y = 0,
\]
where
$B^+_\rho(x) := \{y\in B_\rho(x)\colon y\cdot \nu_v(x) > 0\}$,
$B^-_\rho(x) := \{y\in B_\rho(x)\colon y \cdot\nu_v(x) < 0\}$
and~$\nu_v(x)$ is the unit normal to the jump set at~$x\in\J_v$.
Since~$v$ takes its values in~$\EE$, we have
\[
 \dist_{\R^\ell}(v^+(x), \, \EE)
 \leq \fint_{B^+_\rho(x)} \abs{v(y) - v^+(x)} \d y
 \to 0 \qquad \textrm{as } \rho\to 0,
\]
so~$v^+(x)\in\EE$ and, in a similar way, $v^-(x)\in\EE$.}

\begin{proof}[Proof of Lemma~\ref{lemma:liftbd}]
 Let~$v^\star\in\SBV(\Omega, \, \EE)$ be a lifting of~$u$
 that satisfies~\eqref{goodlifting}, as given by Lemma~\ref{lemma:goodlifting}.
 Let~$v\in\SBV(\Omega, \, \EE)$ be an arbitrary lifting
 of the same map~$u$. Let~$\Phi\colon\EE\to\EE$ be as in~\eqref{Phi}.
 Since~$u\in W^{1,1}(\Omega, \, \NN)$
 has no jump set, the chain rule for BV functions
 (see e.g.~\cite[Theorem~3.96]{AmbrosioFuscoPallara})
 implies that $\Phi(v^+(x)) = \Phi(v^-(x))$,
 $\Phi((v^\star)^+(x)) = \Phi((v^\star)^-(x))$
 at a.e.~$\H^1$-a.e. point~$x\in\J_{v}$, $x\in\J_{v^\star}$ respectively.
 Since the traces at a jump point are different by definition, we must have
 \begin{equation} \label{genlift0}
  v^-(x) = \Phi(v^+(x)), \qquad (v^\star)^-(x) = \Phi((v^\star)^+(x))
 \end{equation}
 at~$\H^1$-a.e. point~$x\in\J_{v}$, $x\in\J_{v^\star}$ respectively.
 We define
 \[
  w := \Xi(v, \, v^\star), \qquad A := \left\{x\in\Omega\colon w(x) = 1\right\},
 \]
 where~$\Xi$ is the function given by Lemma~\ref{lemma:Xi}.
 Since both~$v$ and~$v^\star$ are lifting of the same map~$u$,
 we must have~$w(x)\in\{-1, \, 1\}$ for almost every~$x\in\Omega$.
 The chain rule implies that~$w\in\SBV(\Omega)$
 and that $x$ is a jump point of~$w$
 if and only if $x$ is a jump point for either~$v$ or~$v^\star$
 --- but not both of them, for otherwise
 \[
  w^+(x)
  = \Xi(v^+(x), \, (v^\star)^+(x))
  \stackrel{\eqref{Xitris}}{=}
   \Xi(\Phi(v^+(x)), \, \Phi((v^\star)^+(x))
  \stackrel{\eqref{genlift0}}{=}
   \Xi(v^-(x), \, (v^\star)^-(x))
  = w^-(x).
 \]
 In other words, we have
 \begin{equation*}
  \J_w = \J_{v}\Delta\J_{v^\star} \qquad \mod\H^1.
 \end{equation*}
 Keeping~\eqref{goodlifting} into account, we obtain
 \begin{equation} \label{genlift1}
  \J_v = \J_{w}\Delta\J_{v^\star}
  = \J_{w}\Delta\bigcup_{j=1}^q L_j \qquad \mod\H^1.
 \end{equation}
 On the other hand, since~$w\in\BV(\Omega, \, \{1, \, -1\})$,
 the set~$A$ has finite perimeter and
 \begin{equation} \label{genlift2}
  \J_w = \Omega\cap\partial^* A \qquad \mod\H^1
 \end{equation}
 (see e.g.~\cite[Example~3.68]{AmbrosioFuscoPallara}).
 Together, \eqref{genlift1} and~\eqref{genlift2}
 imply that~$A$ satisfies~\eqref{genlifting}.
 Conversely, given~$A\subseteq\Omega$
 of finite perimeter, we define
 \[
  v(x) := \begin{cases}
       v^\star(x)       &\textrm{if } x\in A \\
       \Phi(v^\star(x)) &\textrm{if } x\in\Omega\setminus A
      \end{cases}
 \]
 and~$w := \Xi(v, \, v^\star)$, as before.
 We know that~$v^\star\in\SBV(\Omega, \, \EE)$,
 hence~$\Phi(v^\star)\in\SBV(\Omega, \, \EE)$
 by the chain rule, and that~$A$ has finite perimeter;
 therefore, we conclude that~$v\in\SBV(\Omega, \, \EE)$
 (see e.g.~\cite[Theorem~3.84]{AmbrosioFuscoPallara}).
 By construction, $v$ is a lifting of~$u$ and, reasoning as before,
 we deduce that~$v$ satisfies~\eqref{genlifting}.
\end{proof}

Keeping Lemma~\ref{lemma:liftbd} into account,
we can reformulate Theorem~\ref{th:minconnrel}
as follows:

\begin{prop} \label{prop:minconn-finiteper}
 For any set~$A\subseteq\Omega$ of finite perimeter, there holds
 \begin{equation} \label{minconnrelineq}
  \H^1(\L_A)
  \geq \mathbb{L}^\Omega(a_1, \, \ldots, \, a_p).
 \end{equation}
 The equality holds if and only if there exists
 another minimal connection~$\{S_1, \, \ldots, \, S_r\}$
 relative to~$\Omega$ such that
 \begin{equation} \label{minconnreleq}
  \L_A = \bigcup_{h=1}^r S_h.
 \end{equation}
\end{prop}

Theorem~\ref{th:minconnrel} is an immediate consequence
of Proposition~\ref{prop:minconn-finiteper}
and Lemma~\ref{lemma:liftbd}.
We will deduce Proposition~\ref{prop:minconn-finiteper}
from a result by Ambrosio, Caselles, Masnou and Morel~\cite{ACMM},
which describes the boundary of planar sets of finite perimeter.
We recall that a set~$C\subseteq\R^2$
is called a Jordan curve if there exists a continuous
function~$\gamma\colon [a, \, b]\to\R^2$
that is injective in~$[a, \, b)$, satisfies~$\gamma(a) = \gamma(b)$
and is such that~$C = \gamma([a, \, b])$.
A rectifiable Jordan curve
is a Jordan curve~$C$ with~$\H^1(C) <+\infty$.
A rectifiable Jordan curve admits a Lipschitz
parametrization (see e.g.~\cite[Lemma~3]{ACMM}).

\begin{theorem}
 \label{th:ACMM}
 Let~$E$ be a subset of~$\R^2$ of finite perimeter.
 Then, there is a decomposition 
 \[
  \partial^* E = \bigcup_{i=0}^{+\infty} C_i
  \mod \H^1
 \]
 into countably many rectifiable Jordan
 curves~$\{C_i\}_{i\geq 0}$ such that 
 \begin{equation} \label{basicallydisj} 
  P(E, \, \R^2) = \sum_{i = 0}^{+\infty} \H^1(C_i).
 \end{equation}
\end{theorem}

Theorem~\ref{th:ACMM} is an immediate consequence of 
of Corollary~1 in~\cite{ACMM}. The latter actually provides 
a much more detailed description of the boundary of planar sets of 
finite perimeter, but the statement we give here is sufficient 
for our purposes.
We can proceed to the proof of 
Proposition~\ref{prop:minconn-finiteper}.

\subsection{Proof of Proposition~\ref{prop:minconn-finiteper}}
\label{sect:minconn-finiteper}

Let~$A\subseteq\Omega$ be a set finite perimeter.
By applying Theorem~\ref{th:ACMM}, we decompose
$\partial^* A$ as a countable union of rectifiable Jordan curves,
$\partial^* A = \cup_{i = 0}^{+\infty} C_i \mod\H^1$,
that satisfy~\eqref{basicallydisj}.
Let~$\gamma_i\colon [0, \, 1]\to\overline{\Omega}$
be Lipschitz parametrisations of~$C_i$, that are
injective on~$[0, \, 1)$ and satisfy~$\gamma_i(0) = \gamma_i(1)$.
For each~$i$, the set~$(\gamma_i)^{-1}\left(\Omega\setminus
\bigcup_{j=1}^q L_j\right)$ is open, hence it can be
written as a countable, disjoint
union of relatively open intervals, say
\[
 (\gamma_i)^{-1}\left(\Omega\setminus\bigcup_{j=1}^q L_j\right)
 = \bigcup_{k = 0}^{+\infty} I_{ik}
\]
Here each~$I_{ik}$ with~$k\geq 1$ is a (possibly empty) open interval,
of the form~$I_{ik} = (a_{ik}, \, b_{ik})$
for some~$0 \leq a_{ik} \leq  b_{ik} \leq 1$,
while~$I_{i0}$ is either empty, the whole
interval~$I_{i0} = [0, \, 1]$, or it has the form
$I^+_{i0} = [0, \, a_{i0}) \cup (b_{i0}, \, 1]$
with~$0 < a_{i0} \leq b_{i0} < 1$.

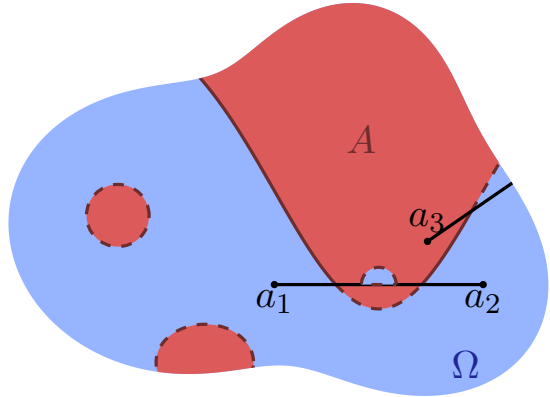
\begin{figure}[t]
  \centering
\resizebox{.35\textheight}{!}{%
\begin{tikzpicture}[scale=1.7, line width = 0.9pt]

\definecolor{lightblue}{RGB}{150,180,255}
\colorlet{darkblue}{lightblue!35!blue!60!black}
\definecolor{myred}{RGB}{220,90,90}
\colorlet{darkred}{myred!50!black}
\definecolor{mypink}{RGB}{245,170,170}
\definecolor{mygray}{RGB}{150,150,150}
\color{darkgray}{mygray!50!black}

\begin{scope}


\def\R{1.25 + 0.25*sin(2*\x) - 0.18*cos(3*\x) + 0.07*sin(4*\x)}

\clip plot[domain=0:360, samples=400]
({-(\R)*sin(\x + 45)}, {(\R)*cos(\x + 45)});

\fill[lightblue] (-2,-2) rectangle (2,2);


\def\BellPath{
  plot[fill=red, domain=-1:1.3, samples=200, variable=\x]
({\x}, {1.2 - 2^(-2*(\x - 0.5)*(\x - 0.5) + 0.8)}) 
}

\fill[myred] \BellPath -- (1.3, 1.2) -- (1.3, 1.3) -- (-1, 1.3) -- cycle;

\draw[darkred] 
 plot[fill=red, domain=-1:0.253, samples=100, variable=\x]
({\x}, {1.2 - 2^(-2*(\x - 0.5)*(\x - 0.5) + 0.8)});
\draw[darkred, dashed] 
 plot[fill=red, domain=0.253:0.747, samples=100, variable=\x]
({\x}, {1.2 - 2^(-2*(\x - 0.5)*(\x - 0.5) + 0.8)});
\draw[darkred] 
 plot[fill=red, domain=0.747:1.033, samples=100, variable=\x]
({\x}, {1.2 - 2^(-2*(\x - 0.5)*(\x - 0.5) + 0.8)});
\draw[darkred, dashed] 
 plot[fill=red, domain=1.033:1.3, samples=20, variable=\x]
({\x}, {1.2 - 2^(-2*(\x - 0.5)*(\x - 0.5) + 0.8)});


\draw[black] (0.78,-0.15) -- (1.5,0.35);

\draw[black] (-0.1,-0.4) -- (1.1,-0.4);

\fill[black] (-0.1,-0.4) circle (0.6pt);
\fill[black] (1.1,-0.4) circle (0.6pt);
\fill[black] (0.78,-0.15) circle (0.6pt);
\node[below] at (-0.1,-0.35) {{\color{black} $a_1$}};
\node[below] at (1.1,-0.35) {{\color{black} $a_2$}};
\node[above] at (0.78,-0.18) {{\color{black} $a_3$}};


\draw[darkred, dashed, fill=lightblue] 
 (0.6,-0.4) arc (0:180:0.1) -- cycle;


\draw[darkred, dashed, line width = 0.9, fill=myred] (-1,0) circle (0.18);


\draw[darkred, dashed, line width = 0.9, fill=myred]
plot[smooth cycle,domain=0:360,samples=200]
({-0.5+0.28*cos(\x)},
 {-0.85+0.22*sin(\x)});

\node[below] at (0.4,0.6) {{\color{darkred} $A$}};
\node[below] at (1,-0.7) {{\color{darkblue} $\Omega$}};

\end{scope}
\end{tikzpicture}
}

  \caption{A set~$A$ (in red) and its arcs.
  The essential arcs are indicated by solid lines, 
  the non-essential ones by dashed lines.}
  \label{fig:nonessentials}
 \end{figure}

For the purposes of this proof, we will call
\emph{arcs} of~$\partial^* A$ all the sets of the form 
$\Gamma_{ik} := \gamma_i\left(\overline{I^+_{ik}}\right)$
for some indices~$i$, $k$,
where~$\overline{I_{ik}}$ denotes the closure of~$I_{ik}$.
An arc can be closed curve or it may have endpoints, which necessarily
belong to~$\partial\Omega\cup\bigcup_{j=1}^q L_j$.
{We will say that an arc of~$\partial^*A$ is \emph{essential}
if and only if it is not a closed curve and 
\begin{enumerate}[label=(\alph*)]
 \item either its endpoints lie on \emph{distinct} 
 line segments~$L_j\neq L_k$, 
 \item or one of them lies on~$\partial\Omega$ and the other
 one lies on a line segment~$L_j$ that does \emph{not}
 touch~$\partial\Omega$.
\end{enumerate}
A non-essential arc of~$\partial^*A$ satisfies
one of the following conditions:
\begin{enumerate}[label=(\roman*)]
 \item either is a closed curve; or
 \item it is not a closed curve and both 
 its endpoints lie on~$\partial\Omega$; or
 \item its endpoints lie on the \emph{same} 
 line segment~$L_j$; or
 \item one of its endpoints lies on~$\partial\Omega$
 and the other one lies on a line segment~$L_j$
 that intersects~$\partial\Omega$.
\end{enumerate}
Figure~\ref{fig:nonessentials} provides an example.

\begin{lemma} \label{lemma:essentials}
 For any~$A\subseteq\Omega$ of finite perimeter,
 $\partial^*A$ has only finitely many essential arcs.
\end{lemma}
\begin{proof}
 Let
 \begin{equation*} 
  \begin{split}
   \eta_0 &:=
   \min\left\{\dist_{\R^2}(L_j, \, L_k)\colon j\neq k\right\}
   \wedge \min\left\{\dist_{\R^2}(L_j, \, \partial\Omega)\colon
    j \textrm{ such that } L_j\subseteq\Omega \right\} \! ,
  \end{split}
 \end{equation*}
 where as usual~$a\wedge b := \min(a, \, b)$
 and~$\dist_{\R^2}$ denotes the Euclidean distance in~$\R^2$. 
 Notice that either there are no essential arcs or
 the number~$\eta_0$ is strictly positive, because
 the line segments~$L_j$ in a minimal connection
 are pairwise disjoint (Lemma~\ref{lemma:mindisj}). 
 All essential arcs of~$\partial^*A$ must have length 
 greater than or equal to~$\eta_0$. On the other hand, 
 since~$A$ has finite perimeter and distinct arcs have
 negligible intersection (because of~\eqref{basicallydisj}
 in Theorem~\ref{th:ACMM}),
 there are only finitely arcs with length greater
 than or equal to~$\eta_0$.
\end{proof}

We first consider the case~$\partial^*A$ only contains
essential arcs, which can be addressed by combinatorial arguments.}
Given two sets~$E$, $F\subseteq\R^2$,
we write~$E\subseteq F\mod\H^1$ as a 
synonym for~$\H^1(E\setminus F) = 0$.

\begin{lemma} \label{lemma:onlyessentials}
 Let~$A\subseteq\Omega$ be a set of finite perimeter
 such that all the arcs of~$\partial^*A$ are essential. 
 Then, there exists finitely many Lipschitz curves
 $f_1$, \ldots, $f_k\colon [0,\, 1]\to\overline{\Omega}$
 that satisfy the following properties:
 \begin{enumerate}[label=(\roman*)]
  \item for each~$h$, either~$f_h(0) = a_i$
  and~$f_h(1) = a_j$ for some~$i\neq j$ or~$f_h(0) = a_i$
  for some~$i$ and~$f_h(1)\in\partial\Omega$;
  \item for each~$i$, there is an
  \emph{odd} number of indices~$h$ such
  that~$a_i$ is an endpoint of the curve~$f_h$;
  \item we have
  \[
   \bigcup_{h=1}^k f_h([0, \, 1]) \subseteq \L_A \mod\H^1,
  \]
  with~$\L_A$ as in~\eqref{L_A};
  \item we have
  \[
   \H^1\!\left(\bigcup_{h=1}^k f_h([0, \, 1])\right)
   = \sum_{h=1}^k \H^1(f_h([0, \, 1)).
  \]
 \end{enumerate}
\end{lemma}
\begin{proof}
 Since~$\Omega\cap\partial^*A$ has finitely many components
 (due to Lemma~\ref{lemma:essentials}), there must be
 finitely many components of~$(\bigcup_{j=1}^q L_j)\setminus
 (\Omega\cap\partial^*A)$.
 These properties allow us to reduce the problem
 to a combinatorial one.
 Let~$\mathscr{G}$ be the finite
 (non-oriented, multi-)graph defined as follows:
 \begin{itemize}
  \item the vertices of~$\mathscr{G}$ are the endpoints
  of all the segments~$L_1$, \ldots, $L_q$ and the endpoints
  of all the arcs in~$\partial^* A$;
  \item the edges of~$\mathscr{G}$
  are all (essential) arcs of~$\partial^* A$
  and all the connected components of
  $(\bigcup_{j=1}^q L_j)\setminus (\Omega\cap\partial^*A)$.
 \end{itemize}
 There might be multiple edges in~$\mathscr{G}$
 that join the same pair of vertices. However, there are
 no loops in~$\mathscr{G}$ (i.e., no edges that join a vertex
 with itself), because no essential arc 
 of~$\partial^* A$ is a closed curve.

 We are interested in the degree of vertices in~$\mathscr{G}$
 --- that is, the number of edges of~$\mathscr{G}$ incident to
 a given vertex. By construction,
 \begin{equation} \label{minconnrel-odddegree}
   a_1, \, \ldots, \, a_p \textrm{ have odd degree in } \mathscr{G}.
 \end{equation}
 Indeed,
 \begin{equation} \label{minconn-degree}
  \begin{split}
   \textrm{degree of } a_i
   &\equiv \left(\textrm{number of line segments~$L_j$ that
   have~$a_i$ as an endpoint}\right) \\
   &\quad + \left(\textrm{number of components
   of~$\partial^* A\phantom{_j}\setminus\{a_i\}$
   that have~$a_i$ as an endpoint}\right) \hspace{-1mm} \mod 2.
  \end{split}
 \end{equation}
 It may happen that a segment~$L_j$
 has~$a_i$ as an endpoint, yet it is not incident to~$a_i$
 in the graph~$\mathscr{G}$, because~$L_j\cap\partial^* A$
 contains a subsegment adjacent to~$a_i$.
 Should this happen, though, both terms at the
 right-hand side of~\eqref{minconn-degree}
 would decrease by one, because~$L_j\cap\partial^* A$
 is not an arc either, and~\eqref{minconn-degree}
 would remain valid anyway.
 Now, by definition of connection relative to~$\Omega$,
 the number of line segments~$L_j$ that have~$a_i$
 as an endpoint is odd. On the other hand,
 the components of~$\partial^* A$ are closed curves
 and any component that passes through~$a_i$
 must intersect~$\partial\Omega\cup\bigcup_{j=1}^q L_j$
 at another point at least, for otherwise there would be
 an arc of~$\partial^* A$ that is a closed curve.
 Therefore, there is an even number of components
 of~$\partial^* A\setminus\{a_i\}$ that have~$a_i$ as an endpoint,
 and~\eqref{minconnrel-odddegree} follows.
 By a similar reasoning, we can also prove that
 \begin{equation} \label{minconnrel-evendegree}
   \textrm{vertices of } \mathscr{G}
   \textrm{ that do not belong to }
   \{a_1, \, \ldots, \, a_p\}\cup\partial\Omega
   \textrm{ have even degree.}
 \end{equation}

 By classical methods in graph theory
 (Fleury's algorithm, see, e.g.~\cite[Theorem~6.5]{Wilson-Graph}),
 there exists a partition of the set of edges of~$\mathscr{G}$
 into disjoint subsets~$\EE_1$, \ldots $\EE_k$,
 where each~$\EE_h$ is a trail (i.e., a sequence
 of distinct edges such that each edge is adjacent to the next one)
 that connects two distinct vertices of~$\mathscr{G}$ with odd degree.
 In particular, the endpoints of~$\EE_h$
 must belong to~$\{a_1, \, \ldots, \, a_p\}\cup\partial\Omega$,
 due to~\eqref{minconnrel-evendegree}.
 Moreover, any~$a_i$ is an endpoint for an odd number
 of trails~$\EE_h$, because of the
 property~\eqref{minconnrel-odddegree}.
 Finally, we disregard all the trails that connect
 two points of~$\partial\Omega$ and define~$f_h\colon [0, \, 1]\to\R^2$
 as a Lipschitz parametrisation of the remaining trails~$\EE_h$.
 Such maps~$f_h$ satisfy all the desired properties.
\end{proof}

{We now modify the set~$A$ in such a way to remove
non-essential arcs from~$\partial^*A$ and show that
this procedure reduces the length of~$\L_A$
(defined by~\eqref{L_A}).
As a preliminary remark, we observe that~$\partial\Omega$
is path-connected, because we have assumed that~$\Omega$
is simply connected and of class~$C^1$.
Moreover,} the geodesic distance~$\dist_{\partial\Omega}(x, \, y)$
between two points~$x$, $y$ in~$\partial\Omega$
--- that is, the length of the shortest curve
contained in~$\partial\Omega$ that connects~$x$ with~$y$ ---
satisfies
\begin{equation} \label{geodist}
 \dist_{\partial\Omega}(x, \, y)
 \leq C_{\Omega} \abs{x - y}
\end{equation}
for some uniform constant~$C_\Omega > 1$
that depends on~$\Omega$ only.
{This inequality is completely analogous
to~\eqref{geodistN}, \eqref{geodistE}.
In the proof of the nex lemma, we will also invoke
the isoperimetric inequality: for any bounded set
of finite perimeter~$E\subseteq\R^2$, there holds
\begin{equation} \label{isoperimetric}
 \abs{E} \leq C \H^1(\partial^* E)^2,
\end{equation}
for some uniform constant~$C$ (see e.g.~\cite[Theorem~3.46]{AmbrosioFuscoPallara}).
Here~$\abs{E}$ is the Lebesgue measure of~$E$.
As is well known, balls are the optimal sets for this inequality,
so the inequality is satisfied with the sharp constant~$C = \frac{1}{4\pi}$,
but this will play no r\^ole in our arguments.}

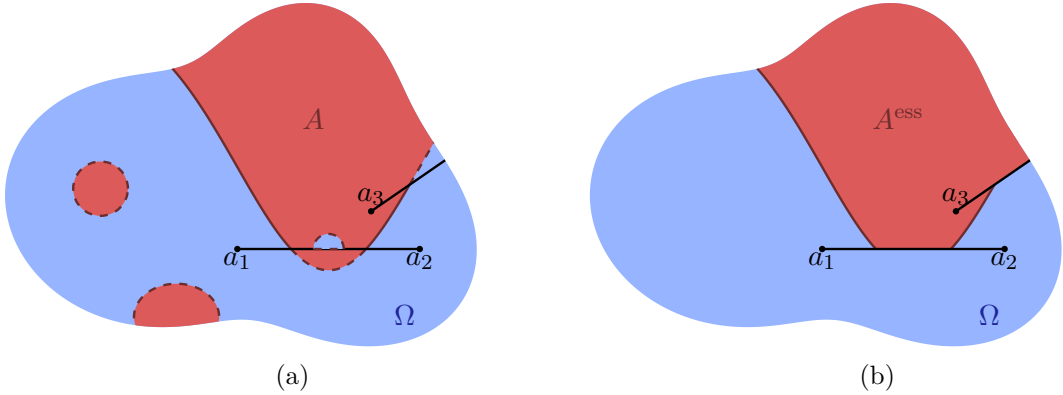
\begin{figure}[t]
 \centering
 \begin{subfigure}{.47\textwidth}
\begin{tikzpicture}[scale=2, line width = 0.9pt]

\definecolor{lightblue}{RGB}{150,180,255}
\colorlet{darkblue}{lightblue!35!blue!60!black}
\definecolor{myred}{RGB}{220,90,90}
\colorlet{darkred}{myred!50!black}
\definecolor{mypink}{RGB}{245,170,170}
\definecolor{mygray}{RGB}{150,150,150}
\color{darkgray}{mygray!50!black}

\begin{scope}


\def\R{1.25 + 0.25*sin(2*\x) - 0.18*cos(3*\x) + 0.07*sin(4*\x)}

\clip plot[domain=0:360, samples=400]
({-(\R)*sin(\x + 45)}, {(\R)*cos(\x + 45)});

\fill[lightblue] (-2,-2) rectangle (2,2);


\def\BellPath{
  plot[fill=red, domain=-1:1.3, samples=200, variable=\x]
({\x}, {1.2 - 2^(-2*(\x - 0.5)*(\x - 0.5) + 0.8)}) 
}

\fill[myred] \BellPath -- (1.3, 1.2) -- (1.3, 1.3) -- (-1, 1.3) -- cycle;

\draw[darkred] 
 plot[fill=red, domain=-1:0.253, samples=100, variable=\x]
({\x}, {1.2 - 2^(-2*(\x - 0.5)*(\x - 0.5) + 0.8)});
\draw[darkred, dashed] 
 plot[fill=red, domain=0.253:0.747, samples=100, variable=\x]
({\x}, {1.2 - 2^(-2*(\x - 0.5)*(\x - 0.5) + 0.8)});
\draw[darkred] 
 plot[fill=red, domain=0.747:1.033, samples=100, variable=\x]
({\x}, {1.2 - 2^(-2*(\x - 0.5)*(\x - 0.5) + 0.8)});
\draw[darkred, dashed] 
 plot[fill=red, domain=1.033:1.3, samples=20, variable=\x]
({\x}, {1.2 - 2^(-2*(\x - 0.5)*(\x - 0.5) + 0.8)});


\draw[black] (0.78,-0.15) -- (1.5,0.35);

\draw[black] (-0.1,-0.4) -- (1.1,-0.4);

\fill[black] (-0.1,-0.4) circle (0.6pt);
\fill[black] (1.1,-0.4) circle (0.6pt);
\fill[black] (0.78,-0.15) circle (0.6pt);
\node[below] at (-0.1,-0.35) {{\color{black} $a_1$}};
\node[below] at (1.1,-0.35) {{\color{black} $a_2$}};
\node[above] at (0.78,-0.18) {{\color{black} $a_3$}};


\draw[darkred, dashed, fill=lightblue] 
 (0.6,-0.4) arc (0:180:0.1) -- cycle;


\draw[darkred, dashed, line width = 0.9, fill=myred] (-1,0) circle (0.18);


\draw[darkred, dashed, line width = 0.9, fill=myred]
plot[smooth cycle,domain=0:360,samples=200]
({-0.5+0.28*cos(\x)},
 {-0.85+0.22*sin(\x)});

\node[below] at (0.4,0.6) {{\color{darkred} $A$}};
\node[below] at (1,-0.7) {{\color{darkblue} $\Omega$}};

\end{scope}
\end{tikzpicture}
  \caption{}
  \label{fig:reduce1}
 \end{subfigure}
 \begin{subfigure}{.47\textwidth}
\begin{tikzpicture}[scale=2, line width = 0.9pt]

\definecolor{lightblue}{RGB}{150,180,255}
\colorlet{darkblue}{lightblue!35!blue!60!black}
\definecolor{myred}{RGB}{220,90,90}
\colorlet{darkred}{myred!50!black}
\definecolor{mypink}{RGB}{245,170,170}
\definecolor{mygray}{RGB}{150,150,150}
\color{darkgray}{mygray!50!black}

\begin{scope}


\def\R{1.25 + 0.25*sin(2*\x) - 0.18*cos(3*\x) + 0.07*sin(4*\x)}

\clip plot[domain=0:360, samples=400]
({-(\R)*sin(\x + 45)}, {(\R)*cos(\x + 45)});

\fill[lightblue] (-2,-2) rectangle (2,2);


\def\BellPath{
  plot[fill=red, domain=-1:1.3, samples=200, variable=\x]
({\x}, {1.2 - 2^(-2*(\x - 0.5)*(\x - 0.5) + 0.8)}) 
}

\fill[myred] \BellPath -- (1.3, 1.2) -- (1.3, 1.3) -- (-1, 1.3) -- cycle;

\draw[darkred] 
 plot[fill=red, domain=-1:0.253, samples=100, variable=\x]
({\x}, {1.2 - 2^(-2*(\x - 0.5)*(\x - 0.5) + 0.8)});
\draw[lightblue, fill=lightblue] 
 plot[fill=red, domain=0.253:0.747, samples=100, variable=\x]
({\x}, {1.2 - 2^(-2*(\x - 0.5)*(\x - 0.5) + 0.8)});
\draw[darkred] 
 plot[fill=red, domain=0.747:1.033, samples=100, variable=\x]
({\x}, {1.2 - 2^(-2*(\x - 0.5)*(\x - 0.5) + 0.8)});

\fill[myred] (0.78,-0.15) -- (1.5,0.35) -- (1.5, 1) -- cycle;


\draw[black] (0.78,-0.15) -- (1.5,0.35);

\draw[black] (-0.1,-0.4) -- (1.1,-0.4);

\fill[black] (-0.1,-0.4) circle (0.6pt);
\fill[black] (1.1,-0.4) circle (0.6pt);
\fill[black] (0.78,-0.15) circle (0.6pt);
\node[below] at (-0.1,-0.35) {{\color{black} $a_1$}};
\node[below] at (1.1,-0.35) {{\color{black} $a_2$}};
\node[above] at (0.78,-0.18) {{\color{black} $a_3$}};

\node[below] at (0.4,0.6) {{\color{darkred} $A^{\mathrm{ess}}$}};
\node[below] at (1,-0.7) {{\color{darkblue} $\Omega$}};

\end{scope}
\end{tikzpicture}

  \caption{}
  \label{fig:reduce2}
 \end{subfigure}
  \caption{Modifying the set~$A$ to remove
  non-essential arcs of~$\partial^*A$, as 
  described in Lemma~\ref{lemma:mindisj}. 
  Left: the original set~$A$; right:
  the modified set~$A^{\mathrm{ess}}$.}
  \label{fig:reduce}
 \end{figure}

\begin{lemma} \label{lemma:nonessentials}
 Let~$A\subseteq\Omega$ be a set of finite perimeter
 such that~$\partial^* A$ has at least one non-essential arc.
 Then, there exists a set of finite perimeter~$A^\ess\subseteq\Omega$
 such that {the arcs of~$\partial^* A^\ess$
 are exactly the essential arcs of~$\partial^* A$} and
 \begin{equation} \label{baba}
  \H^1(\mathscr{L}_{A^\ess}) < \H^1(\mathscr{L}_A).
 \end{equation}
\end{lemma}
\begin{proof}
 {The strategy of the proof is to
 add or remove suitable sets to~$A$, one for each non-essential
 arc of~$\partial^*A$, so as to delete all the non-essential arcs,
 as illustrated in Figure~\ref{fig:reduce}.
 We split the proof into several steps.}

 \begin{step}
  {Let~$\Gamma$ be a non-essential arc of~$\partial^*A$.
  We extend~$\Gamma$ to a Jordan curve~$J(\Gamma)\supseteq\Gamma$
  by distinguishing several cases, as follows.
  \begin{enumerate}[label=(\roman*)]
   \item If~$\Gamma$ is a closed curve,
   then we set~$J(\Gamma) := \Gamma$.

   \item If~$\Gamma$ is a non-closed curve
   with endpoints~$x$, $y$ on~$\partial\Omega$,
   we define~$J(\Gamma)$ by concatenating~$\Gamma$
   with a Lipschitz, one-to-one curve
   $J_{\partial\Omega}(\Gamma)\subseteq\partial\Omega$
   that joins~$y$ to~$x$ in such a way that
   \begin{equation} \label{lengthpartialOmega1}
    \H^1(J_{\partial\Omega}(\Gamma))
    \leq C_{\Omega} \abs{x - y}
    \leq C_{\Omega} \H^1(\Gamma).
   \end{equation}
   Such a curve~$J_{\partial\Omega}(\Gamma)$ exists,
   because of~\eqref{geodist}.

   \item If~$\Gamma$ is a non-closed curve
   with endpoints~$x$, $y$ on the same line segment~$L_j$,
   we first define~$S(\Gamma)$ as the line segment
   with the same endpoints~$x$, $y$ as~$\Gamma$,
   then define~$J(\Gamma)$ by concatenating~$\Gamma$
   with~$S(\Gamma)$ (and travelling along~$S(\Gamma)$
   in the opposite direction).

   \item If~$\Gamma$ is a non-closed curve
   with endpoints~$x\in L_j$, $y\in\partial\Omega$,
   then~$L_j$ itself must have an endpoint on~$\partial\Omega$,
   otherwise~$\Gamma$ would be an essential arc.
   Say, for instance, $\{z\} = L_j\cap\partial\Omega$.
   In this case, we first define $S(\Gamma)$ as the line segment
   from~$x$ to~$z$. Then, we take a Lipschitz,
   one-to-one curve~$J_{\partial\Omega}(\Gamma)$ from~$y$ to~$z$
   such that
   \begin{equation} \label{lengthpartialOmega2}
    \H^1(J_{\partial\Omega}(\Gamma))
    \leq C_{\Omega} \abs{y - z}
    \leq C_{\Omega} \left(\H^1(\Gamma) + \H^1(S(\Gamma))\right) \! .
   \end{equation}
   (Such a curve exists, thanks to~\eqref{geodist}.)
   Finally we define a Jordan curve~$J(\Gamma)$ by concatenating~$\Gamma$
   with~$J_{\partial\Omega}(\Gamma)$ and~$S(\Gamma)$,
   with suitable orientations.
  \end{enumerate}
  In all cases (i)--(iv) above, we have
  \begin{equation} \label{nonewbd}
   J(\Gamma) \setminus \Gamma
   \subseteq \partial\Omega\cup\bigcup_{j=1}^q L_j
  \end{equation}
  by construction.}
 \end{step}

 \begin{step}
  {In cases~(iii) and~(iv), we claim that
  \begin{equation} \label{nonessentials0}
   \H^1(S(\Gamma)) < \H^1(\Gamma).
  \end{equation}
  In the case of~(iii),
  the inequality~\eqref{nonessentials0} is immediate
  because straight lines minimise length.
  In the case of~(iv), suppose (towards a contradiction)
  that~$\abs{x - y} \leq \abs{x - z}$,
  where~$x$, $y$ are the endpoints of~$S(\Gamma)$
  and~$z\in L_j\cap \partial\Omega$, as above.
  Let~$a_i$ be the endpoint of~$L_j$ that is not on~$\partial\Omega$.
  The triangle inequality implies
  \[
   \abs{a_i - y}
   \leq \abs{a_i - x} + \abs{x - y}
   \leq \abs{a_i - x} + \abs{x - z} = \abs{a_i - z}\!.
  \]
  In fact, the equality $\abs{a_i - y} = \abs{a_i - z}$
  cannot hold, for otherwise all the previous inequalities
  would be equalities; we would then have
  $\abs{x - y} = \abs{x - z}$
  and the points~$a_i$, $x$, $y$, $z$,
  would all be aligned, so that
  \[
   \abs{a_i - y}
   = \abs{\abs{a_i - x} - \abs{x - y}}
   < \abs{a_i - x} + \abs{x - z} = \abs{a_i - z}\!,
  \]
  a contradiction.
  Therefore, we must have $\abs{a_i - y} < \abs{a_i - z}$.
  But then, by replacing~$L_j$
  with the straight line segment~$S_j$ of endpoints~$a_i$,
  $y$ (or a smaller line segment
  $S_j^\prime\subseteq S_j\cap\overline{\Omega}$, if necessary),
  we obtain a connection for~$\{a_1, \, \ldots, \, a_p\}$
  relative to~$\Omega$ with total length smaller than
  the one of~$\{L_1, \, \ldots, \, L_q\}$.
  This is again a contradiction, which completes
  the proof of~\eqref{nonessentials0}.
  Combining~\eqref{nonessentials0} with~\eqref{lengthpartialOmega1}
  and~\eqref{lengthpartialOmega2}, we obtain
  \begin{equation} \label{lengthJ}
   \H^1(J(\Gamma)) \leq C \H^1(\Gamma)
  \end{equation}
  in all cases~(i)--(iv), for some uniform constant~$C$
  that depends only on~$\Omega$.}
 \end{step}

 \begin{step}
  {Let~$\{\Gamma_n\}_{n\geq 0}$
  be an enumeration of all non-essential arcs of~$\partial^*A$,
  $J_n := J(\Gamma_n)$, and~$E_n$ be the interior of~$J_n$.
  By construction, we have~$J_n\subseteq\overline\Omega$,
  so~$E_n\subseteq\overline\Omega$.
  We define a sequence of sets~$\{A_n\}_{n\geq 0}$ as follows:
  \[
   \begin{cases}
    A_{0} := A \\
    A_{n} := A_{n-1}\triangle E_{n-1} \qquad \textrm{for all } n\geq 1.
   \end{cases}
  \]
  Lemma~\ref{lemma:boundary} implies that all the sets~$A_n$
  have finite perimeter and
  $\partial^* A_n = \partial^* A_{n-1}\triangle J_{n-1}$ $\mod\H^1$,
  for any~$n\geq 1$.
  Taking~\eqref{nonewbd} into account and reasoning by induction,
  we deduce that
  \begin{equation} \label{nonewbd-n}
   \partial^* A_n\cap \left(\Omega\setminus\bigcup_{j=1}^q L_j\right)
   = \partial^* A \setminus \left(\bigcup_{k=0}^{n-1} \Gamma_k\right)
   \mod\H^1
  \end{equation}
  for all~$n\geq 1$. The proof of~\eqref{nonewbd-n}
  depends on the fact that distinct arcs~$\Gamma_n$, $\Gamma_m$
  have~$\H^1$-null intersection, which is a consequence
  of~\eqref{basicallydisj} in Theorem~\ref{th:ACMM}.
  The same statement, combined with~\eqref{lengthJ},
  implies that
  \begin{equation} \label{CauchyBV1}
   \sum_{n=0}^{+\infty} \H^1(J_n)
   \leq C \sum_{n=0}^{+\infty} \H^1(\Gamma_n) < +\infty,
  \end{equation}
  and the isoperimetric inequality~\eqref{isoperimetric} implies
  \begin{equation} \label{CauchyBV2}
   \sum_{n=0}^{+\infty} \abs{E_n}
   \leq C \sum_{n=0}^{+\infty} \H^1(\Gamma_n)^2 < +\infty.
  \end{equation}
  Let~$\chi_n$ be the indicator function of~$A_n$
  (i.e.~$\chi_n := 1$ on~$A_n$ and~$\chi_n := 0$ otherwise).
  The inequalities~\eqref{CauchyBV1} and~\eqref{CauchyBV2}
  prove that~$\{\chi_n\}_{n\geq 0}$ is a Cauchy sequence in~$\BV(\R^2)$,
  so it converges to a limit~$\chi$. Since BV-convergence implies convergence
  pointwise almost everywhere at least along a subsequence,
  the limit function satisfies~$\chi(x)\in\{0, \, 1\}$ for~a.e.~$x\in\R^2$,
  so (possibly after modification on a negligible set)
  it is itself the indicator function of a set
  of finite perimeter, which we call~$A^\ess$.
  By passing to the limit~$n\to+\infty$ in~\eqref{nonewbd-n},
  we see that~$\partial^*A^\ess \cap (\Omega\setminus\cup_{j=1}^q L_j)$
  consists exactly of the essential arcs of~$\partial^*A$.
  In this procedure, we may have also deleted
  some portions of~$\partial^* A \cap \cup_{j=1}^q L_j$
  of the form~$S(\Gamma_n)$ for some~$n$
  and it might happen that the
  length of~$\cup_{j=1}^q L_j \setminus \partial^* A^\ess$
  is \emph{strictly larger} than the length of
  $\cup_{j=1}^q L_j \setminus \partial^* A$.
  However, this possible increase of length
  is more than compensated by the removal of all non-essential
  arcs, because of~\eqref{nonessentials0}. Therefore,
  the inequality~\eqref{baba} follows.}
  \qedhere
 \end{step}
\end{proof} 

\begin{proof}[Proof of Proposition~\ref{prop:minconn-finiteper}]
 Let~$A\subseteq\Omega$ be a set of finite perimeter.
 Combining Lemma~\ref{lemma:onlyessentials}
 with Lemma~\ref{lemma:nonessentials},
 we find Lipschitz curves
 $f_1$, \ldots, $f_k\colon [0,\ , 1]\to\overline{\Omega}$
 that satisfy the following properties:
 \begin{enumerate}[label=(\roman*)]
  \item for each~$h$, either~$f_h(0) = a_i$
  and~$f_h(1) = a_j$ for some~$i\neq j$ or~$f_h(0) = a_i$
  for some~$i$ and~$f_h(1)\in\partial\Omega$;
  \item for each~$i$, there is an
  \emph{odd} number of indices~$h$ such
  that~$a_i$ is an endpoint of the curve~$f_h$;
  \item we have
  \begin{equation*}
   \sum_{h=1}^k \H^1\!\left(f_h([0, \, 1])\right)
   \leq \H^1(\L_A)
  \end{equation*}
  and the inequality is strict unless
  all the arcs of~$\partial^* A$ are essential.
  (We recall that~$\L_A$ is defined by~\eqref{L_A}.)
 \end{enumerate}
 Let~$S_h$ be the (closed) line segment of
 endpoints~$f_h(0)$ and~$f_h(1)$. If~$S_h\subseteq\overline{\Omega}$
 for all~$h$, then $\{S_1, \, \ldots, \, S_k\}$ is a connection
 for~$\{a_1, \, \ldots, \, a_p\}$ relative to~$\Omega$.
 If a line segment~$S_h$ is not entirely contained in~$\overline{\Omega}$
 and, say, $S_h$ connects a point~$a_i$ with a point of~$\partial\Omega$,
 then we discard~$S_h$ and replace it with~$S_h^\prime\subseteq S_h$,
 the connected component of~$S_h\cap\overline{\Omega}$ that contains~$a_i$.
 In a similar way, if~$S_h\not\subseteq\overline{\Omega}$
 and~$S_h$ connects~$a_i$ with~$a_j$, we replace~$S_h$
 with the connected components~$S_h^\prime$, $S_h^{\prime\prime}$
 of~$S_h\cap\overline{\Omega}$ that contain~$a_i$, $a_j$.
 Either way, after relabelling the segments,
 we construct a connection~$\{S_1, \, \ldots, \, S_r\}$
 for~$\{a_1, \, \ldots, \, a_p\}$ such that
 \begin{equation} \label{blah}
  \H^1(\L_A)
  \geq \sum_{h=1}^k \H^1\!\left(f_h([0, \, 1])\right)
  \geq \sum_{h=1}^r \H^1(S_h).
 \end{equation}
 Then, the inequality~\eqref{minconnrelineq}
 follows immediately from the definition~\eqref{minconnrel}
 of~$\mathbb{L}^\Omega$. If~\eqref{blah}
 reduces to an equality, then all the arcs of~$\partial^* A$
 are essential and Lemma~\ref{lemma:onlyessentials} gives
 \begin{equation} \label{minconncurves}
   \bigcup_{h=1}^k f_h\left([0, \, 1]\right)
   \subseteq \L_A \mod \H^1.
 \end{equation}
 Moreover, equality in~\eqref{blah} implies that
 each~$f_h$ parametrises a straight line segment
 and the inclusion in~\eqref{minconncurves} is actually
 an equality, modulo~$\H^1$-null sets.
 In other words, equality in~\eqref{blah} implies that~$\partial^* A$
 has the form~\eqref{minconnreleq}. This completes the proof.
\end{proof}

\section{Relative minimal connections in a model of ferronematics}
\label{sect:minimisers}

{In this section, we apply Theorem~\ref{th:minconnrel},
along with the arguments from~\cite{CanevariMajumdarStroffoliniWang},
to study the asymptotic behaviour of minimisers~$(\Q^\star_\eps, \, \M^\star_\eps)$
to the energy functional~\eqref{eq:Feps}, subject to ``mixed''
boundary conditions, as in~\eqref{bcbis}, \eqref{hp:bcbis}.
In particular, we will prove Theorem~\ref{th:ferronematics}.
The arguments largely follows the ones in~\cite{CanevariMajumdarStroffoliniWang}.
We only focus on the points which require some adaptation.}


\subsection{Preliminary estimates and notation}

{Several properties of the potential~$f_\eps$
defined by~\eqref{eq:potential}
are listed in~\cite[Lemma~3.1]{CanevariMajumdarStroffoliniWang}.
Here, we only mention that for all sufficiently small~$\eps > 0$
and all~$(\Q, \, \M)\in\Sz\times\R^2$, there holds
\begin{equation} \label{potential_comparison}
  \frac{1}{\eps^2} f_\eps(\Q, \, \M)
  \geq \frac{1}{8\eps^2}(\abs{\Q}^2 - 1)^2
   - \beta^2\abs{\M}^4
\end{equation}
(see Equation~(3.2) in~\cite{CanevariMajumdarStroffoliniWang}).}

\paragraph{The Euler-Lagrange equations and the maximum principle.}

Minimisers~$(\Q^\star_\eps, \, \M^\star_\eps)$ of the functional~\eqref{eq:Feps}
are solutions to the Euler-Lagrange equations
\begin{align}
  -&\Delta\Q^\star_\eps + \dfrac{1}{\eps^2}(\abs{\Q^\star_\eps}^2 - 1)\Q^\star_\eps
  - \dfrac{\beta}{\eps}\left(\M^\star_\eps\otimes\M^\star_\eps
  - \dfrac{\abs{\M^\star_\eps}^2}{2}\I\right)  = 0  \label{EL-Q} \\
  -&\Delta\M^\star_\eps + \dfrac{1}{\eps^2}(\abs{\M^\star_\eps}^2 - 1)\M^\star_\eps
  - \dfrac{2\beta}{\eps^2}\Q^\star_\eps\M^\star_\eps = 0. \label{EL-M}
\end{align}
The following is a standard consequence of
the equations~\eqref{EL-Q}--\eqref{EL-M}.

\begin{lemma} \label{lemma:max}
 The maps~$\Q^\star_\eps$, $\M^\star_\eps$ are smooth inside~$\Omega$
 and of class~$C^2$ up to the boundary of~$\Omega$.
 Moreover, there exist 
 a constant~$C_\beta$, depending only on $\beta$,
 and a constant $C_{\beta,\Omega}$, depending only on $\beta$ and $\Omega$,
 such that
 \begin{align}
  \norm{\Q^\star_\eps}_{L^\infty(\Omega)}
   + \norm{\M^\star_\eps}_{L^\infty(\Omega)} &\leq {C_\beta} \label{max-QM} \\
  \norm{\nabla\Q^\star_\eps}_{L^\infty(\Omega)}
   + \norm{\nabla\M^\star_\eps}_{L^\infty(\Omega)}
   &\leq \frac{{C_{\beta,\Omega}}}{\eps}. \label{max-gradients}
 \end{align}
\end{lemma}
\begin{proof}[Proof of Lemma~\ref{lemma:max}]
 Elliptic regularity theory 
 implies that~$\Q_\eps$ and~$\M_\eps$ are smooth in~$\Omega$
 and of class~$C^2$ up to~$\partial\Omega$.
 We focus on the proof of~\eqref{max-QM};
 the gradient bounds~\eqref{max-gradients} follow from~\eqref{max-QM}
 from elliptic estimates, by reasoning along the lines
 of~\cite[Lemma~A.2]{BBH0}.
 As an intermediate step towards~\eqref{max-QM}, we show that
 \begin{equation} \label{max1}
  \max_{\overline{\Omega}} \abs{\M_\eps}^2
  \leq m_\eps := 1 + {\sqrt{2}}\beta \max_{\overline{\Omega}}\abs{\Q_\eps}
 \end{equation}
 Indeed, by taking the scalar product of~\eqref{EL-M}
 against~$\M_\eps$, we obtain
 \begin{equation*}
  -\frac{1}{2}\Delta\left(\abs{\M_\eps}^2\right)
  + \abs{\nabla\M_\eps}^2
  + \dfrac{1}{\eps^2}(\abs{\M_\eps}^2 - 1)\abs{\M_\eps}^2
  - \dfrac{2\beta}{\eps^2}\Q_\eps\M_\eps\cdot\M_\eps = 0.
 \end{equation*}
 and hence {(observing that $\Q \M \cdot \M \leq \frac{1}{\sqrt{2}} \abs{\Q} \abs{\M}^2$ for any $\Q \in \Sz$ and $\M \in \R^2$)}
 \begin{equation} \label{max2}
  -\frac{1}{2}\Delta\left(\abs{\M_\eps}^2\right)
   \leq - \dfrac{\abs{\M_\eps}^2}{\eps^2}
   \left(\abs{\M_\eps}^2 - 1 - {\sqrt{2}}\beta\abs{\Q_\eps}\right)
 \end{equation}
 at every point of~$\Omega$.
 Now, let~$x_0\in\overline{\Omega}$
 be a maximum point for~$\abs{\M_\eps}^2$
 and suppose, towards a contradiction,
 that~$\abs{\M_\eps(x_0)} > m_\eps$.
 Then, the right-hand side of~\eqref{max2} attains
 a strictly negative value at the point~$x_0$
 (and, by continuity of~$(\Q_\eps, \, \M_\eps)$,
 it is strictly negative in a
 neighbourhood~$V\subseteq\overline{\Omega}$ of~$x_0$).
 If~$x_0$ lies in the interior of~$\Omega$,
 we obtain a contradiction
 because~$-\Delta(\abs{\M_\eps}^2)(x_0) \geq 0$.
 If $x_0\in\partial\Omega$, then Hopf's lemma (applied on~$V$)
 implies that~$\partial_\nnu(\abs{\M_\eps}^2)(x_0) > 0$,
 which contradicts the boundary conditions~\eqref{bcbis}.
 Therefore, \eqref{max1} is proved.

 Now, \eqref{max-QM} follows by~\eqref{max1}
 by the maximum principle. Indeed, by taking the scalar
 product of~\eqref{EL-Q} against~$\Q_\eps$, we deduce
 \begin{equation} \label{maxQ}
  -\frac{1}{2}\Delta\left(\abs{\Q_\eps}^2\right)
  \leq -\dfrac{1}{\eps^2}\left(\abs{\Q_\eps}^4 - \abs{\Q_\eps}^2
  - {\frac{\beta}{\sqrt{2}}}\,\eps\abs{\Q_\eps}\abs{\M_\eps}^2\right)
 \end{equation}
 inside~$\Omega$. Let~$x_1\in\overline{\Omega}$
 be a maximum point for~$\abs{\Q_\eps}^2$.
 If~$x_1\in\partial\Omega$, then $\abs{\Q_\eps(x_1)}^2 = 1$
 because of the assumption~\eqref{hp:bcbis} on the boundary datum.
 If~$x_1\in\Omega$, then~\eqref{maxQ} and~\eqref{max1}
 imply that
 \begin{equation} \label{max4}
  \abs{\Q_\eps(x_1)}^3
  \leq \abs{\Q_\eps(x_1)} + {\frac{\beta}{\sqrt{2}}}\,\eps\abs{\M_\eps(x_1)}^2
  \leq \left(1 + {\beta^2\,\eps}\right)\abs{\Q_\eps(x_1)} + {\frac{\beta}{\sqrt{2}}}\,\eps
 \end{equation}
 and hence, $\abs{\Q_\eps}$ is uniformly bounded, in terms of~$\beta$ only.
\end{proof}

\paragraph{Notation on the Ginzburg-Landau functional.}
{For the~$\Q$-component, we can lift results
from the Ginzburg-Landau literature, thanks to the isometric
isomorphism~$\Sz\to\R^2$ given by
\begin{equation} \label{eq:small-q}
 \Sz \ni \Q \mapsto \q := \sqrt{2}\left(Q_{11},\, Q_{12} \right) \in \R^2.
\end{equation}
In particular, this allows us to borrow some terminology
from~\cite{BBH}. Let~$\Qb\colon\partial\Omega\to\NN$
be a boundary datum of class~$C^2$, oriented by a
$C^2$-map~$\n_\bd\colon\partial\Omega\to\S^1$,
as in~\eqref{hp:bcbis}. A direct computation
shows that the vector field~$\q_{\bd} = ((q_{\bd})_1, \, (q_{\bd})_2)$
associated with~$\Q_{\bd}$ as in~\eqref{eq:small-q}
satisfies~$(q_{\bd})_1 = (n_{\bd})_1^2 - (n_{\bd})_2^2$,
$(q_{\bd})_2 = 2 (n_{\bd})_1 (n_{\bd})_2$
or equivalently, in complex notation,
\begin{equation*} 
 \q_\bd = \n_\bd^2.
\end{equation*}
In particular, if~$d$ denotes the boundary degree of~$\n_\bd$,
then~$\q_\bd$ has degree~$2d$.
Let~$a_1, \, \ldots, \, a_{2\abs{d}}$ be distinct point in~$\Omega$.}
We say that~$\Q^\star\colon\Omega\to\NN$ is a \emph{canonical harmonic map} with singularities at~$(a_1, \, \ldots, \, a_{2\abs{d}})$ and boundary datum~$\Qb$ if it is smooth in~$\Omega\setminus\{a_1, \, \ldots, \, a_{2\abs{d}}\}$, continuous in~$\overline{\Omega}\setminus\{a_1, \, \ldots, \, a_{2\abs{d}}\}$, has non-orientable singularities of degree~$\sign{(d)}/2$ at each point~$a_j$, satisfies the boundary condition~$\Q^\star = \Qb$ on~$\partial\Omega$ and the equation
\[
  \partial_j \left(Q^\star_{11} \, \partial_j Q^\star_{12}
   - Q^\star_{12} \, \partial_j Q^\star_{11}\right) = 0
\]
in the sense of distributions in~$\Omega$.
Such a map exists and is unique, see~\cite[Theorem~I.5, Remark~I.1]{BBH}. The canonical harmonic map belongs to~$W^{1,2}_{\mathrm{loc}}(\Omega\setminus\{a_1, \, \ldots, \, a_{2\abs{d}}\}, \, \NN)$ and the limit
\begin{equation} \label{W}
 \mathbb{W}(a_1, \, \ldots, \, a_{2\abs{d}})
 := \lim_{\sigma\to 0} \left(\frac{1}{2}
  \int_{\Omega\setminus\bigcup_{j=1}^{2\abs{d}} B_\sigma(a_j)}
  \abs{\nabla\Q^\star}^2 \, \d x - 2\pi\abs{d}\abs{\log\sigma}\right)
\end{equation}
exists and is finite (see~\cite[Theorem~I.8]{BBH}).
Following the terminology in~\cite{BBH}, the function~$\mathbb{W}$
is called the \emph{renormalised energy}.
We also define, for~$\eps>0$,
\[
 \gamma(\eps) := \inf\left\{
  \int_{B_1} \left(\frac{1}{2}\abs{\nabla u}^2
   + \frac{1}{4\eps^2} \left(\abs{u}^2 - 1\right)^2 \right)\d x \colon
  u\in W^{1,2}(B_1, \, \C), \ u(x) = x \textrm{ for } x\in \partial B_1\right\}
\]
and
\begin{equation} \label{core_energy}
 \gamma_* := \lim_{\eps\to 0} \left(\gamma(\eps) - \pi\abs{\log\eps}\right)
\end{equation}
The limit~$\gamma_*$ exists, is finite and
strictly positive~\cite[Lemma~III.3]{BBH}.

\paragraph{Allen-Cahn structure for the~$\M$-component.}
{In order to highlight an Allen-Cahn structure
on the variable~$\M_\eps$, as in~\cite{CanevariMajumdarStroffoliniWang},
we rely on a change of variable.
Let $G \subseteq \Omega$ be a simply connected domain 
with smooth boundary $\partial G$,
and let $(\Q_\eps,\,\M_\eps)\in W^{1,2}(\Omega, \, \Sz)
\times W^{1,2}(\Omega, \, \R^2)$ be a pair satisfying
\begin{gather}
	\int_G \left( \frac{1}{2}\abs{\nabla \Q_\eps}^2 + \frac{1}{4\eps^2}\left( \abs{\Q_\eps}^2-1\right)^2 \right) \,{\d}x \leq \Lambda \abs{\log\eps}, \label{eq:logbound+lift-assumptions-1}\\
	\abs{\Q_\eps(x)} \geq \frac{1}{2},\quad \abs{\M_\eps(x)} \leq \Lambda \qquad
	\mbox{for almost any } x \in G,\label{eq:logbound+lift-assumptions-2}
\end{gather}
where~$\Lambda$ is some positive constant that does not depend on $\eps$.
By the spectral theorem and lifting results~\cite{BethuelZheng, BethuelChiron, BallZarnescu}, we can write
\begin{equation}\label{eq:polar-dec-Q-G}
	\Q_\eps = \frac{\abs{\Q_\eps}}{\sqrt{2}}
	\left( \n_\eps \otimes \n_\eps - \m_\eps \otimes \m_\eps \right)
	\qquad \mbox{in } G,
\end{equation}
where $(\n_\eps,\,\m_\eps)$ is an orthonormal frame of eigenvectors
of $\Q_\eps$ in $G$, with $\n_\eps \in W^{1,2}(G,\,\mathbb{S}^1)$ and
$\m_\eps \in W^{1,2}(G,\,\mathbb{S}^1)$.
Then, we define a map~$\u_\eps\colon G \to \R^2$ component-wise,
as $\u_\eps := ( (u_\eps)_1,\,(u_\eps)_2 )$ with
\begin{equation} \label{eq:def-u}
 (u_\eps)_1 := \M_\eps \cdot \n_\eps, \qquad
 (u_\eps)_2 := \M_\eps \cdot \m_\eps.
\end{equation}
This change of variable allows us to effectively
decouple the energy. Indeed, we can rewrite the
functional~\eqref{eq:Feps} (localised to~$G$) as
\begin{equation}\label{eq:Feps-decoupling}
 \F_\eps(\Q_\eps,\,\M_\eps;\,G) = \int_G \left( \frac{1}{2} \abs{\nabla \Q_\eps}^2
 + g_\eps(\Q_\eps)\right)\,{\d}x + \int_G \left( \frac{\eps}{2}\abs{\nabla \u_\eps}^2 + \frac{1}{\eps}h(\u_\eps) \right)\,{\d}x + R_\eps
\end{equation}
(see~\cite[Proposition~3.2]{CanevariMajumdarStroffoliniWang}), where
\begin{gather}
 g_\eps(\Q) := \frac{1}{4\eps^2}\left( \abs{\Q}^2 -1 \right)^2 -
  \frac{2 \kappa_*}{\eps}\left( \abs{\Q}-1 \right) + \kappa_*^2, \label{eq:g} \\
 h(\u) := \frac{1}{4}\left( \abs{\u}^2 - 1 \right)^2
  -\frac{\beta}{\sqrt{2}} \left( u_1^2 - u_2^2 \right)
  + \frac{\beta^2 + \sqrt{2}\beta}{2} \label{eq:h},
\end{gather}
and~$R_\eps$ is a remainder term, such that
\begin{equation} \label{R_eps}
 R_\eps\to 0 \qquad \textrm{as } \eps\to 0
\end{equation}
(cf.~\cite[(3.9) and Lemma~4.10]{CanevariMajumdarStroffoliniWang}).
The functions~$g_\eps$, $h$ are both nonnegative
\cite[Lemma~3.3 and Lemma~3.4]{CanevariMajumdarStroffoliniWang}.
Moreover, $h$ has two non-degenerate minima at
\begin{equation}\label{eq:def-upm}
 \u_{\pm} := \left( \pm\left(\sqrt{2}\beta + 1\right)^{1/2}, 0 \right),
\end{equation}
and the optimal cost for a transition
between~$\u_-$ and~$\u_+$ can be evaluated explicitly: we have
\begin{equation} \label{c_beta}
 \begin{split}
  c_\beta
  &:= \inf\left\{ \int_0^1 \sqrt{2 h(\u(t))}\abs{\u^\prime(t)}\d t\colon
    \u\in W^{1,1}([0, \, 1], \, \R^2), \ \u(0) = \u_-, \ \u(1) = \u_+\right\} \\
  &= \frac{2\sqrt{2}}{3} \left(\sqrt{2}\beta + 1\right)^{3/2}.
 \end{split}
\end{equation}
The interplay between the Ginzburg-Landau
and the Allen-Cahn energies leads to the following
definition of a modified renormalised energy:
given distinct points~$a_1$, \ldots, $a_{2\abs{d}}$ in~$\Omega$,
we define
\begin{equation} \label{W_beta}
 \mathbb{W}_\beta^\Omega(a_1, \, \ldots, \, a_{2\abs{d}})
 := \mathbb{W}(a_1, \, \ldots, \, a_{2\abs{d}})
 + c_\beta \, \mathbb{L}_\beta^\Omega(a_1, \, \ldots, \, a_{2\abs{d}}),
\end{equation}
where~$\mathbb{W}$ is the Ginzburg-Landau reneormalised energy,
as in~\eqref{W}, and~$\mathbb{L}_\beta^\Omega$ is the length of a relative minimal
connection, as defined in~\eqref{minconnrel}.}

\subsection{Proof of Theorem~\ref{th:ferronematics}}

{The next step towards the proof of Theorem~\ref{th:ferronematics}
is the construction of admissible competitors for the minimisation
problem that satisfy sharp energy bounds from above.}

\begin{prop} \label{prop:gammalimsup}
 Let~$a_1, \, \ldots, \, a_{2\abs{d}}$ be distinct points in~$\Omega$.
 Then, there exist maps~$\Q_\eps\in W^{1,2}\left(\Omega, \, \Sz\right)$,
 $\M_\eps\in W^{1,2}\left(\Omega, \, \R^2\right)$ that satisfy~$\Q_\eps = \Qb$
 on~$\partial\Omega$ and
 \begin{equation*}
  \F_\eps (\Q_\eps, \, \M_\eps)\le 2 \pi\abs{d} \abs{\log \eps}
  +\mathbb{W}_\beta^\Omega\left(a_1,\cdots, a_{2\abs{d}}\right)
  + 2 \abs{d} \gamma_* + \o_{\eps\to 0}(1),
 \end{equation*}
 where~$\mathbb{W}_\beta$ and~$\gamma_*$ are
 as in~\eqref{W_beta}, \eqref{core_energy} respectively.
\end{prop}
\begin{proof}
 This result follows along the lines
 of~\cite[Proposition~4.19]{CanevariMajumdarStroffoliniWang}.
 First,  we construct a suitable~$\Q_\eps$ with non-orientable
 singularities at the points~$a_1$, \ldots, $a_{2\abs{d}}$,
 by following existing arguments in the Ginzburg-Landau literature
 (see e.g.~\cite{BBH, AlicandroPonsiglione}).
 Next, we define a vector field~$\widetilde{\M}_\eps\in\SBV(\Omega, \, \R^2)$
 of constant norm, such that $\widetilde{\M}_\eps(x)$
 is an eigenvector of~$\Q_\eps(x)$ at almost every~$x\in\Omega$
 and the jump set of~$\widetilde{\M}_\eps$
 is a minimal connection~$\{L_1, \, \ldots, \, L_q\}$
 for~$\{a_1, \, \ldots, \, a_{2\abs{d}}\}$ relative to~$\Omega$.
 The existence of~$\widetilde{\M}_\eps$ follows from Lemma~\ref{lemma:goodlifting}.
 Now, all that remains to do is to define~$\M_\eps$
 as a regularisation of~$\widetilde{\M}_\eps$ along its jump set.
 This can be achieved 
 by solving an optimal profile problem `\`a la Modica-Mortola'~\cite{ModicaMortola}.
 {The change of variable~\eqref{eq:def-u} turns out to be useful
 in this respect: we first construct suitable maps~$\u_\eps$
 near each line segment~$L_j$, in such a way that~$\u_\eps$
 is an optimal transition profile for~\eqref{c_beta}
 along the transversal direction to~$L_j$;
 then, we define the corresponding~$\M_\eps$
 as~$\M_\eps := (u_\eps)_1 \n_\eps + (u_\eps)_2 \m_\eps$,
 where~$(\n_\eps, \, \m_\eps)$ is an eigenframe for~$\Q_\eps$,
 as in~\eqref{eq:polar-dec-Q-G}.
 This is possible because Lemma~\ref{lemma:mindisj}
 guarantees that the line segments~$L_j$ are pairwise disjoint and
 contained in~$\Omega$, except possibly at their endpoints.
 The details of the construction are carried
 out in~\cite{CanevariMajumdarStroffoliniWang}.}
 The only difference is that some of the~$L_j$'s
 may touch the boundary of~$\Omega$ at an endpoint.
 In this case, the construction carries over
 (Equation~(4.95) in~\cite{CanevariMajumdarStroffoliniWang}
 still makes sense if~$L_1$ has an endpoint on~$\partial\Omega$),
 but there will be error terms in the estimates
 (Equation~(4.96) and~(4.97)) due to boundary effects.
 However, the regularisation only affects a neighbourhood
 of~$L_j$ of thickness~$\sigma_\eps$, with~$\sigma_\eps\to 0$
 as~$\eps\to 0$. Since the boundary of~$\Omega$ is of class~$C^2$,
 the errors we introduce because of the boundary
 tend to zero as~$\eps\to 0$.
\end{proof}

{In the next lemma, we collect useful estimates
that follow from the Ginzburg-Landau literature.}

\begin{lemma} \label{lemma:potboundmin}
 Let~$(\Q_\eps^\star, \, \M_\eps^\star)$ be a minimiser of the functional~$\F_\eps$ subject to the boundary condition~$\Q = \Qb$ on~$\partial\Omega$, where~$\Qb$ satisfies~\eqref{hp:bcbis}.
 Let~$d$ be the degree of~$\Qb$. Then, $(\Q_\eps^\star, \, \M_\eps^\star)$ satisfies
 \begin{equation} \label{mainmin1}
  \F_\eps (\Q_\eps^\star, \, \M_\eps^\star)\le 2 \pi\abs{d} \abs{\log \eps}
  +\mathbb{W}_\beta^\Omega(a_1,\cdots, a_{2\abs{d}})
  + 2 \abs{d} \gamma_* + \o_{\eps\to 0}(1)
 \end{equation}
 for any distinct points~$a_1, \, \ldots, \, a_{2\abs{d}}$ in~$\Omega$.
 Moreover, there exists an~$\eps$-independent constant~$C$ such that
 \begin{equation} \label{potboundmin}
  \int_\Omega \left( \frac{\eps}{2} \abs{\nabla\M^\star_\eps}^2
    + \frac{1}{\eps^2} f_\eps(\Q^\star_\eps, \, \M^\star_\eps) \right){\d}x \leq C.
 \end{equation}
\end{lemma}
\begin{proof}
 The estimate~\eqref{mainmin1} is an immediate consequence of Proposition~\ref{prop:gammalimsup}.
 To prove~\eqref{potboundmin}, we apply existing results
 from the Ginzburg-Landau literature.
 First, we identify~$\Q^\star_\eps$
 with a map~$\q^\star_\eps\colon\Omega\to\R^2$,
 as in~\eqref{eq:small-q}.
 From~\eqref{potential_comparison}, \eqref{max-QM},
 and~\eqref{mainmin1}, we deduce
 \begin{equation} \label{mainmin2}
  \begin{split}
   \int_\Omega \left( \frac{1}{2}
    \abs{\nabla \q^\star_\eps}^2 + \frac{1}{8\eps^2}
    \left(\abs{\q^\star_\eps}^2 - 1 \right)^2\right){\d}x
   &\leq \F_\eps(\Q^\star_\eps,\,\M^\star_\eps)
    + \beta^2 \int_\Omega \abs{\M^\star_\eps}^4 \,{\d}x \\
   &\leq 2\pi\abs{d} \abs{\log\eps} + C,
  \end{split}
 \end{equation}
 where~$C$ is a constant that does not depend on~$\eps$.
 As a consequence \cite[Theorem~1.1]{DelPinoFelmer},
 there exists an~$\eps$-independent constant~$C$
 such that
 \begin{equation} \label{mainmin3}
  \frac{1}{\eps^2} \int_\Omega
    \left(\abs{\Q^\star_\eps}^2 - 1 \right)^2{\d}x
  = \frac{1}{\eps^2} \int_\Omega
    \left(\abs{\q^\star_\eps}^2 - 1 \right)^2{\d}x
  \leq C.
 \end{equation}
 Moreover,
 we can apply estimates such as~\cite[Theorem~2]{Sandier}
 or~\cite[Theorem~1.1]{Jerrard} to~$\q^\star_\eps$ and,
 keeping~\eqref{mainmin3} into account, we deduce
 \begin{equation} \label{mainmin4}
  \begin{split}
   \frac{1}{2} \int_\Omega \abs{\nabla\Q^\star_\eps}^2 {\d}x
   = \frac{1}{2} \int_\Omega \abs{\nabla\q^\star_\eps}^2 {\d}x
   \geq 2\pi\abs{d}\abs{\log\eps} - C,
  \end{split}
 \end{equation}
 again for some~$\eps$-independent~$C$.
 By comparing~\eqref{mainmin1} with~\eqref{mainmin4},
 we obtain the uniform bound~\eqref{potboundmin}.
\end{proof}

\begin{proof}[Proof of Theorem~\ref{th:ferronematics}]
 The estimates in Lemma~\ref{lemma:potboundmin}, together with
 classical results in the Ginzburg-Landau literature (see e.g~\cite[Theorem~2.4]{Lin96}, \cite[Proposition~1.1]{Lin99}, \cite[Theorems~1.2 and~1.3]{Jerrard}, \cite[Theorem~1]{Sandier})
 imply that (a non-relabelled subsequence of) $\Q_\eps^\star$ converges, strongly in~$W^{1,p}(\Omega)$ for~$p<2$, to a map~$\Q^\star$ with
 non-orientable singularities at some
 points~$a^\star_1$, \ldots, $a^\star_{2\abs{d}}$ in~$\Omega$.
 {The proof that~$\Q^\star$ is the canonical harmonic map
 with singularities at~$a^\star_1$, \ldots, $a^\star_{2\abs{d}}$
 is exactly as in~\cite[Proposition~4.9]{CanevariMajumdarStroffoliniWang}.}

 {In order to obtain compactness for~$\M_\eps^\star$,
 we apply the change of variable described in~\eqref{eq:def-u}.
 Classical arguments in the Ginzburg-Landau theory
 (along the lines of~\cite[Theorem~III.3]{BBH})
 imply that there exists a finite set~$X^\star$,
 with~$\{a_1^\star, \, \ldots, \, a_{2\abs{d}}^\star\}\subseteq X^\star
 \subseteq \overline{\Omega}$,
 such that~$\abs{\Q_\eps^\star}\to 1$
 locally uniformly in~$\Omega\setminus X^\star$.
 (In fact, with a bit more work it is possible to show that
 $X^\star = \{a_1^\star, \, \ldots, \, a_{2\abs{d}}^\star\}$
 in our case, but we will not need this equality
 to complete the proof.)
 Let~$G$ be a smooth, simply connected domain
 whose closure is contained in~$\Omega\setminus X^\star$.
 We have~$\abs{\Q_\eps^\star}\geq 1/2$
 in~$G$ for all~$\eps$ small enough, so that we
 may define~$\u_\eps^\star$ as in~\eqref{eq:def-u}.
 From~\eqref{eq:Feps-decoupling} and~\eqref{R_eps}, we know that
 \begin{equation*}
  \int_G \left(\frac{\eps}{2}\abs{\nabla \u_\eps^\star}^2
   + \frac{1}{\eps}h(\u_\eps^\star) \right)\, \d x
  \leq \int_G \left(\frac{\eps}{2}\abs{\nabla \M_\eps^\star}^2
   + \frac{1}{\eps^2} f_\eps(\Q_\eps^\star, \, \M_\eps^\star) \right) \d x
   + \mathrm{o}_{\eps\to 0}(1).
 \end{equation*}
 The uniform bound~\eqref{potboundmin} gives
 \begin{equation*}
  \int_G \left(\frac{\eps}{2}\abs{\nabla \u_\eps^\star}^2
   + \frac{1}{\eps}h(\u_\eps^\star) \right)\, \d x
  \leq C,
 \end{equation*}
 for some constant~$C$ that does not depend on~$\eps$, $G$.
 We can now apply compactness results for the vectorial Modica-Mortola functional
 (see e.g.~\cite{Baldo} or~\cite[Theorems~3.1 and~4.1]{FonsecaTartar})
 to obtain convergence~$\u_\eps^\star\to\u^\star$ in~$L^1(G)$,
 at least along a non-relabelled subsequence.
 We define~$\M^\star := u_1^\star \n^\star + u_2 \m^\star$,
 where~$(\n^\star, \, \m^\star)$ is an eigenframe for~$\Q^\star$.
 The very same arguments as in~\cite[Proposition~4.11]
 {CanevariMajumdarStroffoliniWang} show that~$\M^\star$
 can be extended to a well-defined map~$\SBV(\Omega, \, \R^2)$,
 that~$\abs{\M^\star} = \left(\sqrt{2}\beta + 1\right)^{1/2}$ a.e.~in~$\Omega$
 and that~$\left(\sqrt{2}\beta + 1\right)^{-1/2}\M^\star$ is a lifting of~$\Q^\star$.}

 We can prove a sharp lower bound
 on the energy for the~$\Q_\eps^\star$ by applying
 $\Gamma$-convergence arguments in the Ginzburg-Landau literature~\cite{AlicandroPonsiglione}.
 Given~$\sigma > 0$, let~$D_\sigma := \bigcup_{x\in X^\star} B_\sigma(x)$.
 For any~$\sigma$ small enough, we have (see~\cite[Lemma~4.18]{CanevariMajumdarStroffoliniWang})
 \begin{equation} \label{gammalininf-core}
  \liminf_{\eps\to 0}\left(\F_\eps(\Q_\eps^\star, \, \M_\eps^\star; \, D_\sigma) - 2\pi\abs{d}\abs{\log\eps}\right) \geq \pi\log\sigma + 2\abs{d}\gamma_* - C\sigma,
 \end{equation}
 where~$C$ is a constant that does not depend on~$\eps$, $\sigma$.
 On the other hand, we have
 \begin{equation} \label{gammalininf-outofcore}
  {\liminf_{\eps\to 0}
  \F_\eps(\Q_\eps^\star, \, \M_\eps^\star; \, \Omega\setminus D_\sigma)
  \geq \frac{1}{2} \int_{\Omega\setminus D_\sigma}
   \abs{\nabla\Q^\star}^2 \, \d x
   + c_\beta \H^1\left(\J_{\M^\star}\cap (\Omega\setminus D_\sigma)\right)}
 \end{equation}
 {exactly as in~\cite[Equation~(4.81)]{CanevariMajumdarStroffoliniWang}.}
 Combining~\eqref{gammalininf-core} with~\eqref{gammalininf-outofcore},
 recalling the definition~\eqref{W}
 of~$\mathbb{W}(a^\star_1, \, \ldots, \, a^\star_{2\abs{d}})$,
 {and passing to the limit as~$\sigma\to 0$,} we obtain
 \begin{equation} \label{mainmin6}
  \begin{split}
   \liminf_{\eps\to 0}
    \big(\mathscr{F}_\eps(\Q^\star_\eps, \, \M^\star_\eps) -
    2\pi\abs{d}\abs{\log\eps} \big)
  \geq \mathbb{W}(a^\star_1, \, \ldots, \, a^\star_{2\abs{d}})
    + c_\beta \H^1(\J_{\M^\star})
    + 2\abs{d}\gamma_*.
  \end{split}
 \end{equation}
 By comparing~\eqref{mainmin6} with the upper bound~\eqref{mainmin1},
 and keeping Theorem~\ref{th:minconnrel} into account,
 we obtain
 \[
  \mathbb{W}^\Omega_\beta\left(a^\star_1, \, \ldots, \, a^\star_{2\abs{d}}\right)
  \leq \mathbb{W}\left(a^\star_1, \, \ldots, \, a^\star_{2\abs{d}}\right)
    + c_\beta \H^1\left(\J_{\M^\star}\right)
  \leq \mathbb{W}^\Omega_\beta\left(a_1, \, \ldots, \, a_{2\abs{d}}\right)
 \]
 for any distinct points~$a_1$, \ldots, $a_{2\abs{d}}$ in~$\Omega$. This proves that
 $(a^\star_1, \, \ldots, \, a^\star_{2\abs{d}})$
 is a minimiser of~$\mathbb{W}^\Omega_\beta$ and, moreover,
 $\H^1(\J_{\M_*}) = \mathbb{L}^\Omega(a_1^\star, \, \ldots, \, a^\star_{2\abs{d}})$.
 Then, Theorem~\ref{th:minconnrel}
 implies that~$\J_{\M_*}$ is carried by a
 minimal connection for~$\{a_1^\star, \, \ldots, \, a^\star_{2\abs{d}}\}$
 relative to~$\Omega$.
 This completes the proof of the theorem.
\end{proof}

\paragraph{Acknowledgments}
The authors are members of GNAMPA-INdAM. Their work has been
partially supported by GNAMPA projects CUP\_E53C22001930001{,} CUP\_E53C23001670001{,}
and CUP\_E53C25002010001.
B.S.'s research
is part of the project ``Geometric Evolution Problems and Shape
Optimizations'', PRIN  Project 2022E9CF89 and  FRA Project-B “VarMoCry”, Variational analysis and modeling of Liquid Crystals, funded by  the University of Naples Federico II.
G.C. is part of the ANR project ``Singularities of energy-minimizing vector-valued maps'', ref. ANR-22-CE40-0006.
Part of the reserach that lead to the present work was carried
out while the authors participated in the summer school
``Variational and PDE Methods in Nonlinear Science'', organised by
the C.I.M.E. Foundaton in Cetraro (Italy), July~2023, and in the workshop
``24w5249 --- Mathematical Analysis of Soft Matter'',
organised by the Banff International Research Station
in Banff (Canada) in July 2024.
F.L.D. would like to thank the University of Verona
for hospitality during the last part of this work.
Likewise, G.C. is grateful to the University of Napoli Federico~II
for hospitality during several visits.

\bibliographystyle{plain}
\bibliography{UnifConv}

\begin{thebibliography}{10}

\bibitem{AlicandroPonsiglione}
R.~Alicandro and M.~Ponsiglione.
\newblock Ginzburg-{L}andau functionals and renormalized energy: a revised
  {$\Gamma$}-convergence approach.
\newblock {\em J. Funct. Anal.}, 266(8):4890--4907, 2014.

\bibitem{Ambrosio-metricBV}
L.~Ambrosio.
\newblock Metric space valued functions of bounded variation.
\newblock {\em Ann. Sc. norm. super. Pisa - Cl. sci}, Ser. 4, 17(3):439--478,
  1990.

\bibitem{ACMM}
L.~Ambrosio, V.~Caselles, S.~Masnou, and J.M. Morel.
\newblock Connected components of sets of finite perimeter and applications to
  image processing.
\newblock {\em J. Eur. Math. Soc.}, 3(1):39--92, 2001.

\bibitem{AmbrosioFuscoPallara}
L.~Ambrosio, N.~Fusco, and D.~Pallara.
\newblock {\em Functions of bounded variation and free discontinuity problems}.
\newblock Oxford Mathematical Monographs. The Clarendon Press, Oxford
  University Press, New York, 2000.

\bibitem{BadalCicalese}
R.~Badal and M.~Cicalese.
\newblock Renormalized energy between fractional vortices with topologically
  induced free discontinuities on 2-dimensional riemannian manifolds.
\newblock {\em Calc. Var. Partial Differential Equations}, 64:83, 2025.

\bibitem{Badal_et_al}
R.~Badal, M.~Cicalese, L.~De~Luca, and M.~Ponsiglione.
\newblock {$\Gamma$}-convergence analysis of a generalized $xy$ model:
  Fractional vortices and string defects.
\newblock {\em Comm. Math. Phys.}, 358(2):705--739, 2018.

\bibitem{Baldo}
S.~Baldo.
\newblock Minimal interface criterion for phase transitions in mixtures of
  {C}ahn-{H}illiard fluids.
\newblock {\em Ann. Inst. H. Poincar\'{e} C Anal. Non Lin\'{e}aire},
  7(2):67--90, 1990.

\bibitem{BallZarnescu}
J.~Ball and A.~Zarnescu.
\newblock Orientability and energy minimization in liquid crystal models.
\newblock {\em Arch. Ration. Mech. Anal.}, 202(2):493--535, 2011.

\bibitem{Bedford}
S.~Bedford.
\newblock Function spaces for liquid crystals.
\newblock {\em Arch. Ration. Mech. Anal.}, 219(2):937--984, 2016.

\bibitem{BellettiniMarzianiScala}
G.~Bellettini, R.~Marziani, and R.~Scala.
\newblock On jump minimizing liftings for $\mathbb{S}^1$‑valued maps and
  connections with {A}mbrosio‑{T}ortorelli‑type {$\Gamma$}‑limits.
\newblock Preprint, arXiv:2505.08731, 2025.

\bibitem{BBH0}
F.~Bethuel, H.~Brezis, and F.~H{\'e}lein.
\newblock Asymptotics for the minimization of a {G}inzburg-{L}andau functional.
\newblock {\em Calc. Var. Partial Differential Equations}, 1(2):123--148, 1993.

\bibitem{BBH}
F.~Bethuel, H.~Brezis, and F.~H\'{e}lein.
\newblock {\em Ginzburg-{L}andau vortices}, volume~13 of {\em Progress in
  Nonlinear Differential Equations and their Applications}.
\newblock Birkh\"{a}user Boston, Inc., Boston, MA, 1994.

\bibitem{BethuelChiron}
F.~Bethuel and D.~Chiron.
\newblock Some questions related to the lifting problem in {S}obolev spaces.
\newblock In {\em Perspectives in nonlinear partial differential equations},
  volume 446 of {\em Contemp. Math.}, pages 125--152. Amer. Math. Soc.,
  Providence, RI, 2007.

\bibitem{BethuelDemengel}
F.~Bethuel and F.~Demengel.
\newblock Extensions for {S}obolev mappings between manifolds.
\newblock {\em Calc. Var. Partial Differential Equations}, 3(4):475--491, 1995.

\bibitem{BethuelZheng}
F.~Bethuel and X.M. Zheng.
\newblock Density of smooth functions between two manifolds in {S}obolev
  spaces.
\newblock {\em J. Funct. Anal.}, 80(1):60--75, 1988.

\bibitem{bisht2019}
K.~Bisht, V.~Banerjee, P.~Milewski, and A.~Majumdar.
\newblock Magnetic nanoparticles in a nematic channel: A one-dimensional study.
\newblock {\em Phys- Rev. E}, 100(1):012703, 2019.

\bibitem{BourgainBrezisMironescu2005}
J.~Bourgain, H.~Brezis, and P.~Mironescu.
\newblock Lifting, degree, and distributional jacobian revisited.
\newblock {\em Comm. Pure Appl. Math.}, 58(4):529--551, 2005.

\bibitem{BrezisCoronLieb}
H.~Brezis, J.M. Coron, and E.H. Lieb.
\newblock Harmonic maps with defects.
\newblock {\em Comm. Math. Phys.}, 107(4):649--705, 1986.

\bibitem{BrezisMironescu}
H.~Brezis and P.~Mironescu.
\newblock {\em Sobolev maps to the circle---from the perspective of analysis,
  geometry, and topology}, volume~96 of {\em Progress in Nonlinear Differential
  Equations and their Applications}.
\newblock Birkh\"{a}user/Springer, New York, [2021] \copyright 2021.

\bibitem{Brochard}
F.~Brochard and P.G. de~Gennes.
\newblock Theory of magnetic suspensions in liquid crystals.
\newblock {\em J. Phys.}, 31(7):691--708, 1970.

\bibitem{Caldereretal}
M.C. Calderer, A.~DeSimone, D.~Golovaty, and A.~Panchenko.
\newblock An effective model for nematic liquid crystal composites with
  ferromagnetic inclusions.
\newblock {\em SIAM J. Appl. Math.}, 74(2):237--262, 2014.

\bibitem{CDS}
G.~Canevari, F.~L. Dipasquale, and B.~Stroffolini.
\newblock The formation of gradient-driven singular structures of codimension
  one and two in two-dimensions: The case study of ferronematics.
\newblock Preprint, arXiv:2505.07506, 2025.

\bibitem{CanevariMajumdarStroffoliniWang}
G.~Canevari, A.~Majumdar, B.~Stroffolini, and Y.~Wang.
\newblock Two-dimensional ferronematics, canonical harmonic maps and minimal
  connections.
\newblock {\em Arch. Ration. Mech. Anal.}, 247(6):Paper No. 110, 61, 2023.

\bibitem{CO-lifting}
G.~Canevari and G.~Orlandi.
\newblock Lifting for manifold-valued maps of bounded variation.
\newblock {\em J. Funct. Anal.}, 278(10):108453, 2020.

\bibitem{ContiCrismaleGarroniMalusa}
S.~Conti, V.~Crismale, A.~Garroni, and A.~Malusa.
\newblock Phase-field approximation of sharp-interface energies accounting for
  lattice symmetry.
\newblock Preprint, arXiv:2601.07497, 2026.

\bibitem{DavilaIgnat}
J.~D\'avila and R.~Ignat.
\newblock Lifting of {BV} functions with values in ${S}^1$.
\newblock {\em C.R. Math. Acad. Sci. Paris}, 337(3):159--164, 2003.

\bibitem{DelPinoFelmer}
M.~del Pino and P.L. Felmer.
\newblock Local minimizers for the {G}inzburg-{L}andau energy.
\newblock {\em Math. Z.}, 225(4):671--684, 1997.

\bibitem{Federer}
H.~Federer.
\newblock {\em Geometric measure theory}.
\newblock Die Grundlehren der mathematischen Wissenschaften, Band 153.
  Springer-Verlag New York, Inc., New York, 1969.

\bibitem{FonsecaTartar}
I.~Fonseca and L.~Tartar.
\newblock The gradient theory of phase transitions for systems with two
  potential wells.
\newblock {\em Proc. Roy. Soc. Edinburgh Sect. A}, 111(1-2):89--102, 1989.

\bibitem{GiaquintaModicaSoucek}
M.~Giaquinta, G.~Modica, and J.~Sou{\v{c}}ek.
\newblock {\em Cartesian currents in the calculus of variations.}, volume~37 of
  {\em Ergebnisse der Mathematik und ihrer Grenzgebiete. 3. Folge. A Series of
  Modern Surveys in Mathematics [Results in Mathematics and Related Areas. 3rd
  Series. A Series of Modern Surveys in Mathematics]}.
\newblock Springer-Verlag, Berlin, 1998.
\newblock Cartesian currents.

\bibitem{GoldmanMerletMillot}
M.~Goldman, B.~Merlet, and V.~Millot.
\newblock A {Ginzburg-Landau} model with topologically induced free
  discontinuities.
\newblock {\em Annales de l'Institut Fourier}, 70(6):2583--2675, 2020.

\bibitem{Hatcher}
A.~Hatcher.
\newblock {\em {A}lgebraic {T}opology}.
\newblock Cambridge University Press, Cambridge, 2002.

\bibitem{IgnatLamy}
R.~Ignat and X.~Lamy.
\newblock Lifting of $\mathbb{{RP}}^{d-1}$-valued maps in {BV} and applications
  to uniaxial {Q}-tensors. with an appendix on an intrinsic {BV}-energy for
  manifold-valued maps.
\newblock {\em Calc. Var. Partial Differential Equations}, 58(2):68, Mar 2019.

\bibitem{Jerrard}
R.L. Jerrard.
\newblock Lower bounds for generalized {G}inzburg-{L}andau functionals.
\newblock {\em SIAM J. Math. Anal.}, 30(4):721--746, 1999.

\bibitem{Lin96}
F.H. Lin.
\newblock Some dynamical properties of {G}inzburg-{L}andau vortices.
\newblock {\em Comm. Pure Appl. Math.}, 49(4):323--359, 1996.

\bibitem{Lin99}
F.H. Lin.
\newblock Vortex dynamics for the nonlinear wave equation.
\newblock {\em Comm. Pure Appl. Math.}, 52(6):737--761, 1999.

\bibitem{Merlet}
B.~Merlet.
\newblock Two remarks on liftings of maps with values into $\mathbb{S}^1$.
\newblock {\em C.R. Math. Acad. Sci. Paris}, 343(7):467--472, 2006.

\bibitem{MLDC}
A.~Mertelj, D.~Lisjak, M.~Drofenik, and M.~Copic.
\newblock Ferromagnetism in suspensions of magnetic platelets in liquid
  crystal.
\newblock {\em Nature}, 504:237–241, 2013.

\bibitem{MironescuRussSire}
P.~Mironescu, E.~Russ, and Y.~Sire.
\newblock Lifting in {B}esov spaces.
\newblock {\em Nonlinear Anal.}, 2019.

\bibitem{MironescuVanSchaftingen}
P.~Mironescu and J.~Van~Schaftingen.
\newblock Lifting in compact covering spaces for fractional {S}obolev mappings.
\newblock {\em Anal. PDE}, 14(6):1851--1871, 2021.

\bibitem{ModicaMortola}
L.~Modica and S.~Mortola.
\newblock Un esempio di {$\Gamma \sp{-}$}-convergenza.
\newblock {\em Boll. Un. Mat. Ital. B (5)}, 14(1):285--299, 1977.

\bibitem{Mucci-DCDS}
D.~Mucci.
\newblock Maps into projective spaces: liquid crystal and conformal energies.
\newblock {\em Discrete Cont. Dyn.-B}, 17(2):597--635, 2012.

\bibitem{Sandier}
\'{E}. Sandier.
\newblock Lower bounds for the energy of unit vector fields and applications.
\newblock {\em J. Funct. Anal.}, 152(2):379--403, 1998.

\bibitem{SchoenUhlenbeck2}
R.~Schoen and K.~Uhlenbeck.
\newblock Boundary regularity and the {D}irichlet problem for harmonic maps.
\newblock {\em J. Differential Geom.}, 18(2):253--268, 1983.

\bibitem{Wilson-Graph}
R.~J. Wilson.
\newblock {\em Introduction to Graph Theory}.
\newblock Pearson Education/Prentice Hall, Harlow, England, 5th edition, 2010.

\end{thebibliography}

\begin{flushright}
\Addresses
\end{flushright}

\end{document}